\documentclass[oneside]{amsart}
\usepackage[backref]{hyperref}
\usepackage{amsmath,color,graphicx,amsfonts,amssymb,latexsym,amsthm,enumerate,comment,mathtools}
\graphicspath{{./img_pdf/}}
\usepackage[hang,small]{caption}
\usepackage[subrefformat=parens]{subcaption}
\usepackage{tikz,tikz-cd}
\usetikzlibrary{intersections,calc,arrows.meta}
\tikzset{point/.style={circle, fill=black, inner sep=1pt, minimum size=3pt}}
\tikzcdset{every label/.append style={font=\small}}

\def\Z{\mathbb{Z} }
\def\E{\mathbb{E} }
\def\K{\mathbb{K} }

\def\R{\mathbb{R} }

\def\L{\mathbb{L} }
\def\M{\mathbb{M} }
\def\W{\mathbb{W} }
\def\C{\mathcal{C} }

\newcommand{\snc}{\mathrm{snc}}

\newcommand{\wsum}[4]{\sum_{#3\in#1}w(#2,#3)w(#3,#4)}

\DeclareMathOperator{\rank}{rank}
\DeclareMathOperator{\im}{im}

\theoremstyle{definition}
\newtheorem{definition}{Definition}
\newtheorem{theorem}{Theorem}[section]
\newtheorem{proposition}[theorem]{Proposition}
\newtheorem{lemma}[theorem]{Lemma}
\newtheorem{corollary}[theorem]{Corollary}
\newtheorem{remark}{Remark}

\newtheorem{example}{Example}
\title{Morse-Bott inequalities on Lefschetz complexes}
\author{Yuto Nishikawa}
\address{Graduate School of Science and Engineering, Saitama University,
255 Shimo-Okubo, Sakura-ku, Saitama-shi, Saitama 338-8570, Japan}
\email{nishikawa.y.492@ms.saitama-u.ac.jp}
\date{}

\begin{document}
\begin{abstract}
We develop a discrete Morse-Bott theory for Lefschetz complexes with real-valued
incidence functions and finitely many cells in each dimension. Our main result is a
reduction procedure, based on four elementary operations called Moves, for computing
the remainder series $R_t$ associated with the Morse-Bott inequality. These Moves
preserve $R_t$, and the procedure recursively constructs, using finitely many Moves in
each dimension, a disjoint union of two-cell elementary blocks with vanishing
Poincar\'{e} series. The resulting decomposition
computes the coefficients of $R_t$ by counting the blocks in the corresponding
dimensions. The same procedure also computes the Betti numbers of the original complex
by counting the cells removed as isolated cells. As consequences, we obtain the
nonnegativity of the coefficients of $R_t$ and the Morse-Bott inequality for Lefschetz
complexes. The method applies in particular to CW complexes, where the nonnegativity is obtained by
identifying each coefficient of $R_t$ with the number of elementary blocks in the
corresponding dimension.
\end{abstract}

\maketitle

\section{Introduction}
 
Morse theory, developed by Morse~\cite{morse1925} in the 1920s, relates the topology
of a smooth manifold to the critical points of smooth functions defined on it.
Bott~\cite{bott54} extended this to \emph{Morse-Bott theory}, where isolated critical
points are replaced by critical submanifolds.  The \emph{Morse-Bott inequality}
expresses the Poincar\'{e} polynomial of the manifold in terms of the Poincar\'{e}
polynomial of each critical submanifold shifted by its Morse index, with a remainder
of the form $(1+t)Q_t$, where $Q_t$ has nonnegative integer coefficients.
 
Forman~\cite{forman98} introduced \emph{discrete Morse theory}, a combinatorial
analogue of Morse theory defined on CW complexes.  A discrete Morse function assigns
values to cells, and a cell is \emph{critical} if it satisfies a condition analogous
to that of an isolated critical point.  The discrete Morse inequality then relates
the numbers of critical cells to the Betti numbers of the complex.
 
Yaptieu~\cite{yaptieu2017} proposed a discrete Morse-Bott theory for CW complexes,
replacing isolated critical cells with more general weakly critical sets.
In~\cite{nishikawa2025DMBT}, the Morse-Bott inequality was established for finite CW
complexes: for such a complex $M$ and a discrete Morse-Bott function $f$ on $M$,
\[
  \sum_{C \in \mathcal{C}_f} P_t(C) = P_t(M) + (1+t)R_t(M, f),
\]
where $\mathcal{C}_f$ denotes the collection of weakly critical sets and
$R_t(M, f)$ is a polynomial with nonnegative integer coefficients.
 
The algebraic framework of CW complexes has been abstracted to the notion of a
\emph{Lefschetz complex}~\cite{kacz2016stable,mrozek2016stable}. We use a degreewise
finite extension of this notion: a triple $(X,\dim,w)$ consisting of a set of cells
whose $k$-dimensional members form a finite set for each $k \geq 0$, together with a
dimension function and a real-valued incidence function
$w\colon X \times X \to \R$ satisfying ``$w^2=0$''.
This formulation retains only the combinatorial boundary structure of a CW complex,
discarding its geometric and topological data, and serves as an algebraic setting
for combinatorial dynamics~\cite{kacz2016stable,mrozek2016stable} and, more
recently, for a theory of cancellations and depth posets in persistent
homology~\cite{edel2024DMFLef}.
 
\medskip
 
The present paper extends discrete Morse-Bott theory to Lefschetz complexes and
introduces a new approach --- based on a collection of \emph{Moves} --- that yields
a concise proof of the nonnegativity of $R_t$.  This provides an alternative to
the proof of~\cite{nishikawa2025DMBT} and makes the argument transparent by
reducing the problem to a disjoint union of elementary blocks.
 
\medskip
 
\noindent\textbf{Discrete Morse-Bott theory on Lefschetz complexes.}
We introduce the notion of a discrete Morse-Bott function on a Lefschetz complex
$\K = (X, \dim, w)$, show that weakly critical sets are subcomplexes,
and that a discrete Morse-Bott function induces a discrete vector field.
We prove the following.
 
\begin{theorem}[Morse-Bott inequality for Poincar\'{e} series]\label{thm:MB}
  Let $f$ be a discrete Morse-Bott function on a Lefschetz complex $\K$.
  Then the remainder series $R_t(\K, f)$ has nonnegative integer coefficients and
  satisfies
  \[
    \sum_{C \in \mathcal{C}_f} P_t(C) = P_t(\K) + (1+t)R_t(\K, f),
  \]
  where $\mathcal{C}_f$ denotes the collection of weakly critical sets of~$f$.
\end{theorem}
 
\noindent\textbf{Moves and decomposition into elementary blocks.}
To prove the nonnegativity of $R_t(\K, f)$, we introduce four
\emph{Moves} --- elementary operations on pairs $(\K, f)$ that preserve
$R_t(\K, f)$ --- and use them to reduce $(\K, f)$ to a disjoint
union of \emph{elementary blocks}.  An elementary block $E(\mu, \kappa)$ is a
two-cell Lefschetz complex with $w(\mu, \kappa) = 1$ and all other incidence numbers
zero; its Poincar\'{e} series vanishes.
 
\begin{theorem}[Elementary block reduction]\label{thm:elementary}
  Let $f$ be a discrete Morse-Bott function on a Lefschetz complex $\K$.
  There exists a Lefschetz complex $\E$, constructed recursively by applying finitely
  many Moves in each dimension, such that $\E$ is a disjoint union of elementary blocks
  and $R_t(\K, f) = R_t(\E, \dim)$.
\end{theorem}

Since each elementary block has vanishing Poincar\'{e} series, comparing the two sides of
the identity for $(\E, \dim)$ shows that $R_t(\K, f)$ counts the elementary blocks of $\E$,
graded by dimension (Corollary~\ref{cor:count}). More precisely, the coefficient of $t^k$
in $R_t(\K,f)$ is the number of elementary blocks whose upper cell has dimension $k+1$.
In particular, the coefficients of $R_t(\K,f)$ are nonnegative, which completes the proof
of Theorem~\ref{thm:MB}. Since a finite CW complex defines a Lefschetz complex as in
Example~\ref{ex:CWasLef}, Theorem~\ref{thm:MB} yields the corresponding Morse-Bott
inequality for finite CW complexes; this specialization is recorded in
Corollary~\ref{cor:MBCW}.

\section{Preliminaries}
In this section, we review the fundamental concepts and properties of Lefschetz complexes. 
\subsection{Fundamental concepts}
\begin{definition}[Lefschetz complex]
Let $X$ be a set, and let $\dim\colon X\to\Z_{\ge0}$ and $w\colon X\times X\to\R$ be functions. Assume that for every $k\in\Z_{\ge0}$, the set $\dim^{-1}(k)$ is finite. A triple $(X,\dim, w)$ is a \textbf{Lefschetz complex} if the following conditions hold:
\begin{enumerate}[(L1)]
    \item For any $\tau,\sigma\in X$ with $w(\tau,\sigma)\ne0$, we have $\dim\tau-\dim\sigma=1$.
    \item For any $\tau,\nu\in X$, we have $\wsum{X}{\tau}{\sigma}{\nu}=0$.
\end{enumerate}
Here the sum in (L2) is finite: by (L1), the term $w(\tau,\sigma)w(\sigma,\nu)$ vanishes unless $\sigma\in\dim^{-1}(\dim\tau-1)$, and this set is finite by assumption.
\end{definition}

An element $\sigma\in X$ is called a cell. If $\dim\sigma=k$, we denote it by $\sigma^k$, and call it a $k$-cell. From now on, let $\K:=(X,\dim, w)$ be a Lefschetz complex. We also write $\sigma\in \K$ for $\sigma\in X$. Throughout the paper, the symbols $\tau$, $\sigma$ and $\nu$ denote cells whose dimensions
decrease by one at each step, so that $w(\tau,\sigma)$ and $w(\sigma,\nu)$ may be nonzero.
The symbols $\mu$ and $\kappa$ denote a pair of cells to be cancelled, the symbol $\eta$
denotes a cell having no nonzero incidence number with any cell, and other cells are
denoted by $\alpha$ and $\beta$.

\begin{example}\label{ex:CWasLef}
    Let $X$ be the set of cells of a CW complex having finitely many cells in each dimension. Then $(X,\dim,w)$ is a Lefschetz complex, where $\dim$ is the cellular dimension function and $w(\tau,\sigma)$ denotes the coefficient of $\sigma$ in the cellular boundary $\partial\tau$.
\end{example}

Condition (L2) implies that a cell connecting $\tau$ to $\nu$ through nonzero incidence numbers is never unique. This is used repeatedly below.

\begin{lemma}\label{lem:incidence}
    For any cells $\tau,\sigma,\nu\in \K$ with $w(\tau,\sigma)w(\sigma,\nu)\ne0$, there is a cell $\widetilde{\sigma}\ne\sigma$ such that $w(\tau,\widetilde{\sigma})w(\widetilde{\sigma},\nu)\ne0$.

\end{lemma}

\begin{proof}
    Suppose that $w(\tau,\widetilde{\sigma})w(\widetilde{\sigma},\nu)=0$ for any cell $\widetilde{\sigma}\ne\sigma$. Then all terms with $\alpha\neq\sigma$ vanish in the sum
    $\sum_{\alpha\in \K} w(\tau,\alpha)w(\alpha,\nu)$.
    Since $\K$ satisfies condition (L2), we obtain
    \[
    0=\wsum{\K}{\tau}{\alpha}{\nu}=w(\tau,\sigma)w(\sigma,\nu)\ne0,
    \]
    which is a contradiction. 
\end{proof}

\begin{definition}
    Let $\{\K_i:=(X_i,\dim_i,w_i)\}_{i\in I}$ be a family of Lefschetz complexes. We define their disjoint union $\bigsqcup_{i\in I}\K_i:=(X_I,\dim_I,w_I)$, where $X_I:=\sqcup_{i\in I}X_i$, the function $\dim_I\colon X_I\to\Z_{\ge0}$ is given by $\dim_I|_{X_i}:=\dim_i$ for each $i\in I$, and the function $w_I\colon X_I\times X_I\to\R$ is given by
    \[
    w_I(\tau,\sigma):=\begin{cases}
      w_i(\tau,\sigma) & \text{if }\tau,\sigma\in X_i\quad\text{for some }i\in I,\\
      0 & \text{otherwise}.
    \end{cases}
    \]
\end{definition}

\begin{lemma}\label{lem:disjoint}
    If $\dim_I^{-1}(k)$ is finite for every $k\in\Z_{\ge0}$, then $\bigsqcup_{i\in I}\K_i$ is
    a Lefschetz complex.
\end{lemma}

\begin{proof}
    Condition (L1) is inherited from each $\K_i$: if $w_I(\tau,\sigma)\neq0$, then
    $\tau,\sigma\in X_i$ for some $i\in I$ and $w_I(\tau,\sigma)=w_i(\tau,\sigma)$, so
    $\dim_I\tau-\dim_I\sigma=1$.

    For (L2), let $\tau,\nu\in X_I$. If $\tau,\nu\in X_i$ for some $i\in I$, then
    $w_I(\tau,\sigma)w_I(\sigma,\nu)=0$ for every $\sigma\in X_I\setminus X_i$, and hence
    \[
    \sum_{\sigma\in X_I}w_I(\tau,\sigma)w_I(\sigma,\nu)
    =\sum_{\sigma\in X_i}w_I(\tau,\sigma)w_I(\sigma,\nu)
    =\sum_{\sigma\in X_i}w_i(\tau,\sigma)w_i(\sigma,\nu)=0
    \]
    by condition (L2) for $\K_i$. If $\tau\in X_i$ and $\nu\in X_j$ with $i\neq j$, then for
    every $\sigma\in X_I$ at least one of $w_I(\tau,\sigma)$ and $w_I(\sigma,\nu)$ vanishes, so the
    sum is zero.
\end{proof}

\begin{definition}
    A subset $A\subset \K$ is a \textbf{subcomplex} of $\K$ if
    $(A,\dim|_A,w|_{A\times A})$ is a Lefschetz complex.
\end{definition}

In fact, condition (L1) and the finiteness assumption are automatically inherited by any subset, so being a subcomplex is a condition on (L2) alone. We record this criterion, which is used repeatedly below.

\begin{proposition}\label{prp:subcpx}
    Let $A$ be a subset of $\K$. The following are equivalent:
    \begin{enumerate}[(1)]
        \item $A$ is a subcomplex of $\K$.
        \item The triple $(A,\dim|_A,w|_{A\times A})$ satisfies condition (L2).
    \end{enumerate}
\end{proposition}

\begin{proof}
    Assume that $A$ is a subcomplex of $\K$. Then assertion (2) holds.

    Conversely, assume that assertion (2) holds. Fix any cells $\tau,\sigma\in A$ with $w(\tau,\sigma)\ne0$. Since $\dim|_A (\alpha)=\dim\alpha$ for any cell $\alpha\in A$, we obtain the following equality:
    \[
    \dim|_A(\tau)-\dim|_A(\sigma)=\dim\tau-\dim\sigma=1.
    \]
    Thus, the triple $(A,\dim|_A,w|_{A\times A})$ satisfies (L1). Moreover, $(\dim|_A)^{-1}(k)$
    is finite for every $k$, being a subset of $\dim^{-1}(k)$, and (L2) holds by assertion (2).
    Hence, assertion (1) holds.
\end{proof}

\begin{definition}
Denote by $C_k(\K)$ the real vector space generated by the $k$-cells of $\K$. Define $C_{-1}(\K):=\{0\}$. For any integer $k\in\Z_{\ge0}$, we define the boundary homomorphism
\[
\partial_k^\K\colon C_k(\K)\to C_{k-1}(\K),
\]
for any $k$-cell $\tau^k\in \K$ by
\[
\partial_k^\K(\tau):=\sum_{\sigma^{k-1}\in \K}w(\tau,\sigma)\sigma,
\]
and extend $\partial_k^\K$ linearly to $C_k(\K)$.
\end{definition}

When the Lefschetz complex $\K$ is clear from the context, we simply write $\partial_k$ instead of $\partial_k^\K$.

\begin{proposition}
    For any integer $k\in\Z_{>0}$, the equality $\partial_{k-1}\circ\partial_k = 0$ holds.
\end{proposition}

\begin{proof}
    Fix any integer $k\in\Z_{>0}$ and any cell $\tau^k\in \K$. By condition (L2), we have
    \[
    \partial_{k-1}\circ\partial_k(\tau)=\sum_{\sigma\in \K}w(\tau,\sigma)\partial_{k-1}(\sigma)=\sum_{\sigma\in \K}w(\tau,\sigma)\sum_{\nu\in \K}w(\sigma,\nu)\nu=\sum_{\nu\in \K}0\nu=0.
    \]
\end{proof}

\begin{definition}
    Let $k\in\Z_{\ge0}$. We define the $k$-th homology group $H_k(\K)$ as follows:
    \begin{align*}
        H_k(\K)&:=Z_k(\K)/B_k(\K),
    \end{align*}
    where $Z_k(\K):=\ker \partial_k^\K,B_k(\K):=\im \partial_{k+1}^\K$. Define the $k$-th Betti number $b_k^\K$ and the Poincar{\'e} series $P_t(\K)$ as follows:
    \begin{align*}
        b_k^\K&:=\rank H_k(\K)\\
        P_t(\K)&:=\sum_{k=0}^\infty b_k^\K t^k
    \end{align*}

\end{definition}

Extending the notation above to $k=-1$, we set $B_{-1}(\K):=\im\partial_0=\{0\}$.

\begin{lemma}\label{lem:disjointsum}
    Let $\{\K_i\}_{i\in I}$ be a family of Lefschetz complexes whose disjoint union
    $\L:=\bigsqcup_{i\in I}\K_i$ is a Lefschetz complex. For every $k\in\Z_{\ge0}$, only
    finitely many $i\in I$ satisfy $\rank C_k(\K_i)\neq0$, and
    \[
    \rank B_k(\L)=\sum_{i\in I}\rank B_k(\K_i),\qquad
    b_k^\L=\sum_{i\in I}b_k^{\K_i}.
    \]
    In particular, $P_t(\L)=\sum_{i\in I}P_t(\K_i)$.
\end{lemma}

\begin{proof}
    Since $\dim^{-1}(k)$ is finite in $\L$ for every $k$, only finitely many summands
    contain cells of a given dimension. If $w(\tau,\sigma)\neq0$ in $\L$, then $\tau$ and
    $\sigma$ lie in the same $\K_i$, so $\partial_k^\L$ maps $C_k(\K_i)$ into
    $C_{k-1}(\K_i)$. Hence the chain complex of $\L$ is the direct sum of the chain
    complexes of the $\K_i$, and both the image of $\partial$ and the homology commute
    with direct sums.
\end{proof}

\begin{lemma}\label{lem:rankid}
    For every $k\in\Z_{\ge0}$,
    \[
    \rank C_k(\K)=\rank B_k(\K)+\rank B_{k-1}(\K)+b_k^\K.
    \]
\end{lemma}

\begin{proof}
    The two short exact sequences
    \begin{align*}
    &0\to Z_k(\K) \to C_k(\K) \to B_{k-1}(\K) \to 0,\\
    &0 \to B_k(\K) \to Z_k(\K) \to H_k(\K) \to 0
    \end{align*}
    give, by the rank-nullity theorem, $\rank C_k(\K)=\rank Z_k(\K)+\rank B_{k-1}(\K)$ and
    $\rank Z_k(\K)=\rank B_k(\K)+b_k^\K$. Adding these gives the assertion.
\end{proof}

Let $f\colon Y\to\R$ be a function. We say $f$ is a function on a Lefschetz complex $\K=(X,\dim,w)$ if $X\subset Y$. We also write $f\colon\K\to\R$ for a function $f\colon Y\to\R$ on $\K$.

This convention allows us to keep the same function $f$ when passing to the smaller
complexes obtained by the Moves of Section~\ref{sec:moves}.

\begin{definition}
    Let $f\colon\K\to\R$ be a function. For any cell $\sigma\in \K$, we define two numbers $U(\sigma)$ and $D(\sigma)$ as follows:
    \[
    U(\sigma):=\#\{\tau\in \K\mid w(\tau,\sigma)\ne0,f(\sigma)\ge f(\tau)\}
    \]
    \[
    D(\sigma):=\#\{\nu\in \K\mid w(\sigma,\nu)\ne0,f(\sigma)\le f(\nu)\}
    \]
\end{definition}

\begin{definition}
        Let $f\colon\K\to\R$ be a function. For any cell $\sigma\in \K$, we define two numbers $U^\snc(\sigma)$ and $D^\snc(\sigma)$ as follows:
    \[
    U^\snc(\sigma):=\#\{\tau\in \K\mid w(\tau,\sigma)\ne0,f(\sigma)> f(\tau)\}
    \]
    \[
    D^\snc(\sigma):=\#\{\nu\in \K\mid w(\sigma,\nu)\ne0,f(\sigma)< f(\nu)\}
    \]
\end{definition}

Here the superscript $\snc$ stands for ``strongly noncritical'': the cells counted by $U^\snc(\sigma)$ and $D^\snc(\sigma)$ are those counted by $U(\sigma)$ and $D(\sigma)$ with a strict inequality of $f$-values. In particular $U^\snc(\sigma)\le U(\sigma)$ and $D^\snc(\sigma)\le D(\sigma)$.

When it is necessary to indicate $\K$ and $f$ explicitly, we write $U_{\K,f}(\sigma)$,
$D_{\K,f}(\sigma)$, $U^\snc_{\K,f}(\sigma)$ and $D^\snc_{\K,f}(\sigma)$. We use the same
convention for $m_k$, $u_k$, $d_k$, $m_k^\snc$, $u_k^\snc$ and $d_k^\snc$.

\begin{definition}
    Let $f\colon\K\to\R$ be a function and $r\in\R$. If the set $\{\sigma\in\K\mid f(\sigma)=r\}$ is nonempty, we call it a \textbf{level set} of $f$. For a level set $L$, the set $\{\sigma\in L\mid U^\snc(\sigma)+D^\snc(\sigma)=0\}$ is called a \textbf{weakly critical set}. Denote by $\mathcal{C}_f$ the family of all nonempty weakly critical sets.
\end{definition}

\begin{remark}\label{rem:wcunique}
    Two weakly critical cells lie in the same weakly critical set if and only if they have
    the same value under $f$. In particular, a weakly critical set is determined by the
    level set containing it.
\end{remark}

When it is necessary to indicate $\K$ explicitly, we write $\mathcal{C}_f$ as $\mathcal{C}_{\K,f}$.

\subsection{Chain homotopy equivalence}

The cancellation of a pair of cells with nonzero incidence number, introduced in
Section~\ref{sec:moves}, replaces a Lefschetz complex by a smaller one whose chain complex
is chain homotopy equivalent to the original one (Lemma~\ref{lem:ChHoMuKappa}). In this
subsection we recall the notions involved and record that such a replacement leaves the
Betti numbers unchanged.

Let $(A_\bullet,\partial_\bullet)$ and $(B_\bullet,d_\bullet)$ be chain complexes. We recall
the following standard notions of homological algebra (see, e.g., \cite[\S2.1]{Hatcher2002top}).
\begin{itemize}
    \item A family of homomorphisms $\varphi_\bullet\colon A_\bullet\to B_\bullet$ is a
    \textbf{chain map} if $\varphi_{n-1}\circ\partial_n=d_n\circ\varphi_n$ for every $n$.
    \item A family of homomorphisms $h_\bullet\colon A_\bullet\to B_{\bullet+1}$ is a
    \textbf{chain homotopy} between chain maps $\varphi,\psi\colon A\to B$ if
    $h_{n-1}\circ\partial_n+d_{n+1}\circ h_n=\varphi_n-\psi_n$ for every $n$. The chain maps
    $\varphi$ and $\psi$ are \textbf{chain homotopic} if such an $h$ exists.
    \item The chain complexes $A$ and $B$ are \textbf{chain homotopy equivalent} if there are
    chain maps $\varphi\colon A\to B$ and $\psi\colon B\to A$ such that $\varphi\circ\psi$ and
    $\psi\circ\varphi$ are chain homotopic to the identities.
\end{itemize}

\begin{lemma}\label{lem:rankhom}
    Let $(A,\partial)$ and $(B,d)$ be chain complexes. If $A$ and $B$ are chain homotopy equivalent, then for every $n$, the $n$-th homology groups are isomorphic. In particular, $\rank H_n(A)=\rank H_n(B)$.
\end{lemma}
\begin{proof}
    This follows immediately from the fact that chain homotopic maps induce the same homomorphism on homology (see, e.g., \cite[Prop.~2.12]{Hatcher2002top}).
\end{proof}

\subsection{Discrete vector fields}

We now introduce discrete vector fields, which formalize the notion of an arrow on a Lefschetz complex.

\begin{definition}
    A set $V\subset X\times X$ is a \textbf{discrete vector field} on $\K$ if it satisfies the following two conditions:
    \begin{enumerate}[(V1)]
        \item For any pair $(\tau,\sigma)\in V$, we have $w(\tau,\sigma)\neq0$.
        \item For any distinct pairs $(\tau,\sigma),(\widetilde{\tau},\widetilde{\sigma})\in V$, we have $\{\tau,\sigma\}\cap\{\widetilde{\tau},\widetilde{\sigma}\}=\emptyset$.
    \end{enumerate}
\end{definition}

In our illustrations, cells of the same dimension are placed horizontally,
and higher-dimensional cells appear above lower-dimensional ones. Throughout
this and the subsequent examples, we draw an edge whenever $w(\tau,\sigma)\neq 0$
and place the label $w(\tau,\sigma)$ near that edge. A pair $(\tau,\sigma)\in V$ is drawn
as an arrow from $\sigma$ to $\tau$, that is, from the cell of lower dimension to the cell
of higher dimension. This is the convention behind the notation $\sigma\to_f\tau$
introduced in Section~\ref{sec:MBfunctions}.

\begin{example}
Figure~\ref{fig:dvf-example} shows, on the same Lefschetz complex, an example of a discrete vector field and an example that is not a discrete vector field. The symbols $\nu_i$, $\sigma_i$, and $\tau_i$ denote cells. Red arrows represent the pairs that belong to the subset $V$.
\end{example}

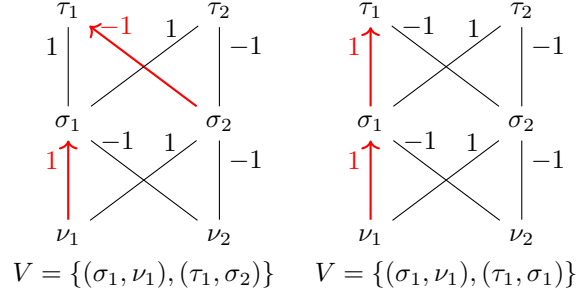
\begin{figure}[htbp]
\centering
\begin{tikzpicture}[scale=1]
    \begin{scope}
      \node (a) at (0,0) {$\nu_1$};
      \node (b) at (0,1.5) {$\sigma_1$};
      \node (c) at (0,3) {$\tau_1$};
      \node (d) at (2,0) {$\nu_2$};
      \node (e) at (2,1.5) {$\sigma_2$};
      \node (f) at (2,3) {$\tau_2$};

      \draw[->,red,thick] (a) -- (b) node[near end, left] {$1$};
      \draw (b) -- (c) node[near end, left] {$1$};
      \draw (d) -- (e) node[near end, right] {$-1$};
      \draw (e) -- (f) node[near end, right] {$-1$};
      \draw (a) -- (e) node[near end, above] {$1$};
      \draw (d) -- (b) node[near end, above] {$-1$};
      \draw[->,red,thick] (e) -- (c) node[near end, above] {$-1$};
      \draw (b) -- (f) node[near end, above] {$1$};

      \node at (1,-0.5) {$V=\{(\sigma_1,\nu_1),(\tau_1,\sigma_2)\}$};
    \end{scope}

    \begin{scope}[shift={(4,0)}]
      \node (a) at (0,0) {$\nu_1$};
      \node (b) at (0,1.5) {$\sigma_1$};
      \node (c) at (0,3) {$\tau_1$};
      \node (d) at (2,0) {$\nu_2$};
      \node (e) at (2,1.5) {$\sigma_2$};
      \node (f) at (2,3) {$\tau_2$};

      \draw[->,red,thick] (a) -- (b) node[near end, left] {$1$};
      \draw[->,red,thick] (b) -- (c) node[near end, left] {$1$};
      \draw (d) -- (e) node[near end, right] {$-1$};
      \draw (e) -- (f) node[near end, right] {$-1$};
      \draw (a) -- (e) node[near end, above] {$1$};
      \draw (d) -- (b) node[near end, above] {$-1$};
      \draw (e) -- (c) node[near end, above] {$-1$};
      \draw (b) -- (f) node[near end, above] {$1$};

      \node at (1,-0.5) {$V=\{(\sigma_1,\nu_1),(\tau_1,\sigma_1)\}$};
    \end{scope}
\end{tikzpicture}
\caption{Two subsets $V\subset X\times X$: the left one is a discrete vector
field, while the right one violates condition~(V2).}
\label{fig:dvf-example}
\end{figure}

Every function $f$ on $\K$ determines two subsets of $X\times X$, which serve as the discrete
counterpart of the negative gradient flow of $f$.

\begin{definition}
    Let $f\colon\K\to\R$ be a function. The sets $-\nabla f$ and $-\nabla_s f$ are defined as follows:
    \begin{align*}
        -\nabla f&:=\{(\tau,\sigma)\in X\times X\mid w(\tau,\sigma)\neq0,f(\tau)\le f(\sigma)\}\\
        -\nabla_s f&:=\{(\tau,\sigma)\in X\times X\mid w(\tau,\sigma)\neq0,f(\tau)< f(\sigma)\}
    \end{align*}
    We call $-\nabla f$ the \textbf{gradient vector field} of $f$, and $-\nabla_s f$ the
    \textbf{strict gradient vector field} of $f$.
\end{definition}

By Proposition~\ref{prp:gvfisdvf} and Proposition~\ref{prp:gvfisdvfMB}, these sets are
indeed discrete vector fields when $f$ is a discrete Morse function and a discrete
Morse-Bott function, respectively.

\begin{remark}\label{rem:degree}
    Let $f\colon\K\to\R$ be a function. Directly from the definitions,
    \[
    U(\sigma)=\#\{\tau\in\K\mid(\tau,\sigma)\in-\nabla f\},\qquad
    D(\sigma)=\#\{\nu\in\K\mid(\sigma,\nu)\in-\nabla f\},
    \]
    and likewise $U^\snc(\sigma)$ and $D^\snc(\sigma)$ count the pairs of $-\nabla_s f$ having
    $\sigma$ as second and as first coordinate, respectively. Thus $U$ and $D$ (resp.\
    $U^\snc$ and $D^\snc$) are the in-degree and the out-degree of $\sigma$ in $-\nabla f$
    (resp.\ in $-\nabla_s f$), viewed as a directed graph on $X$.
\end{remark}

\begin{definition}
    Let $V$ be a discrete vector field on $\K$. The set
    \[
    F(V):=\{\sigma\in \K \mid \sigma \text{ belongs to no pair in } V\}
    \]
    is called the \textbf{fixed-point set} of $V$.
\end{definition}

The fixed-point set consists of the cells left unpaired by $V$; for gradient vector fields
it will play the role of the set of critical points.

\section{Discrete Morse functions}

In this section, we introduce discrete Morse functions on Lefschetz complexes and
establish an identity for the Poincar\'{e} series of $\K$. The nonnegativity of the
remainder term is not proved here; it is obtained in Section~\ref{sec:moves} as
Corollary~\ref{cor:M}.

\begin{definition}[Discrete Morse function]
    A function $f\colon\K\to\R$ is a \textbf{discrete Morse function} on $\K$ if for any cell $\sigma\in \K$, $U(\sigma)\le1$ and $D(\sigma)\le 1$.
\end{definition}

The definition bounds $U(\sigma)$ and $D(\sigma)$ separately, but they cannot both be
equal to $1$: condition (L2), through Lemma~\ref{lem:incidence}, forces the following
stronger inequality.

\begin{lemma}\label{lem:one}
  For any discrete Morse function $f\colon\K\to\R$ and for any cell $\sigma\in \K$, we have the following inequality
  \[
  U(\sigma)+D(\sigma)\le1.
  \]
\end{lemma}

\begin{proof}
  Suppose that there is a cell $\sigma\in \K$ with $U(\sigma)=1$ and $D(\sigma)=1$. By the definition of $U(\sigma)$ and $D(\sigma)$, we have two cells $\tau,\nu\in \K$ with $w(\tau,\sigma)w(\sigma,\nu)\neq0$ and $f(\tau)\le f(\sigma)\le f(\nu)$. By Lemma~\ref{lem:incidence}, we have $\tilde{\sigma}\in \K$ such that $\sigma\neq\tilde{\sigma}$ and $w(\tau,\tilde\sigma)w(\tilde\sigma,\nu)\neq0$.

  Since $w(\tau,\sigma)\neq0$ and $f(\tau)\le f(\sigma)$, the cell $\sigma$ is counted in $D(\tau)$, so $D(\tau)\ge1$. Since $f$ is a discrete Morse function, $D(\tau)\le1$, and hence $D(\tau)=1$ with $\sigma$ as the unique cell counted in $D(\tau)$. Since $w(\tau,\tilde\sigma)\neq0$ and $\tilde\sigma\neq\sigma$, the cell $\tilde\sigma$ is not counted in $D(\tau)$, that is, $f(\tau)\le f(\tilde\sigma)$ fails; hence $f(\tau)>f(\tilde\sigma)$.

  Similarly, since $w(\sigma,\nu)\neq0$ and $f(\sigma)\le f(\nu)$, the cell $\sigma$ is counted in $U(\nu)$, so $U(\nu)\ge1$. Since $U(\nu)\le1$, we have $U(\nu)=1$ with $\sigma$ as the unique cell counted in $U(\nu)$. Since $w(\tilde\sigma,\nu)\neq0$ and $\tilde\sigma\neq\sigma$, the cell $\tilde\sigma$ is not counted in $U(\nu)$, that is, $f(\nu)\ge f(\tilde\sigma)$ fails; hence $f(\tilde\sigma)>f(\nu)$.

  Combining these, we obtain $f(\tau)>f(\tilde\sigma)>f(\nu)$. Therefore, we obtain the following inequality
  \[
  f(\tau) > f(\tilde\sigma) > f(\nu) \ge f(\sigma) \ge f(\tau),
  \]
  which is a contradiction.
\end{proof}

\begin{definition}
    Let $f\colon\K\to\R$ be a discrete Morse function. A cell $\sigma\in \K$ is \textbf{critical} for $f$ if $U(\sigma)+D(\sigma)=0$.
\end{definition}

\begin{proposition}\label{prp:gvfisdvf}
    Let $f\colon\K\to\R$ be a discrete Morse function. The set $-\nabla f$ is a discrete vector field on $\K$.
\end{proposition}

\begin{proof}
    For any pair $(\tau,\sigma)\in -\nabla f$, we have $w(\tau,\sigma)\neq0$ by the definition of $-\nabla f$. Therefore, condition (V1) holds.
    
    For (V2), note that $(\tau,\sigma)\in-\nabla f$ means precisely that $\sigma$ is counted
    in $D(\tau)$ and $\tau$ is counted in $U(\sigma)$. Let $(\tau,\sigma)$ and
    $(\widetilde\tau,\widetilde\sigma)$ be distinct pairs in $-\nabla f$ sharing a cell
    $\alpha$. If $\alpha=\tau=\widetilde\tau$, then $\sigma\neq\widetilde\sigma$ and both are
    counted in $D(\alpha)$; if $\alpha=\sigma=\widetilde\sigma$, then
    $\tau\neq\widetilde\tau$ and both are counted in $U(\alpha)$; otherwise $\alpha$ is the
    first coordinate of one pair and the second coordinate of the other, so $U(\alpha)\ge1$
    and $D(\alpha)\ge1$. In every case $U(\alpha)+D(\alpha)\ge2$, contradicting
    Lemma~\ref{lem:one}. Therefore, condition (V2) holds, and $-\nabla f$ is a discrete
    vector field on $\K$.
\end{proof}

\begin{remark}
    Let $f\colon\K\to\R$ be a discrete Morse function. A cell is a fixed point of $-\nabla f$ if and only if it is critical for $f$.
\end{remark}

\begin{example}
Figure~\ref{fig:gvf-example} illustrates how the gradient vector field
$-\nabla f$ is defined for a discrete Morse function on a Lefschetz complex.
The diagram on the left shows a Lefschetz complex $\K$, and the middle diagram
displays a discrete Morse function $f$ together with the pairs forming its
gradient vector field $-\nabla f$. In this example, the fixed points of $-\nabla f$
coincide with the critical cells of $f$.
The diagram on the right shows another discrete Morse function $f'$ on the same complex $\K$, for which $-\nabla f'$ is empty; in this case, every cell is a fixed point of $-\nabla f'$, and therefore every cell is critical for $f'$.
\end{example}

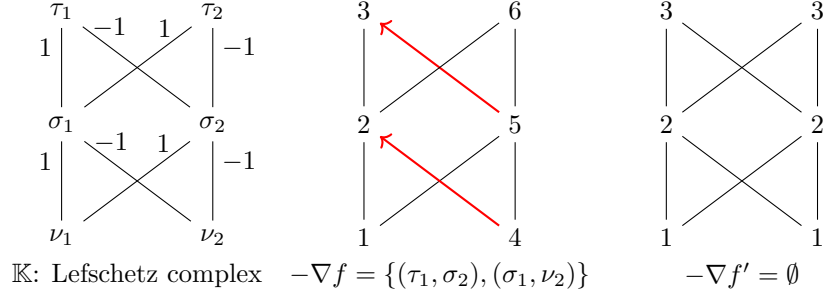
\begin{figure}[htbp]
\centering
\begin{tikzpicture}[scale=1]

    \begin{scope}[shift={(-4,0)}]
      \node (a) at (0,0) {$\nu_1$};
      \node (b) at (0,1.5) {$\sigma_1$};
      \node (c) at (0,3) {$\tau_1$};
      \node (d) at (2,0) {$\nu_2$};
      \node (e) at (2,1.5) {$\sigma_2$};
      \node (f) at (2,3) {$\tau_2$};

      \draw (a) -- (b) node[near end, left] {$1$};
      \draw (b) -- (c) node[near end, left] {$1$};
      \draw (d) -- (e) node[near end, right] {$-1$};
      \draw (e) -- (f) node[near end, right] {$-1$};
      \draw (a) -- (e) node[near end, above] {$1$};
      \draw (d) -- (b) node[near end, above] {$-1$};
      \draw (e) -- (c) node[near end, above] {$-1$};
      \draw (b) -- (f) node[near end, above] {$1$};

      \node at (1,-0.5) {$\K$: Lefschetz complex};
    \end{scope}

    \begin{scope}
      \node (a) at (0,0) {$1$};
      \node (b) at (0,1.5) {$2$};
      \node (c) at (0,3) {$3$};
      \node (d) at (2,0) {$4$};
      \node (e) at (2,1.5) {$5$};
      \node (f) at (2,3) {$6$};

      \draw (a) -- (b);
      \draw (b) -- (c);
      \draw (d) -- (e);
      \draw (e) -- (f);
      \draw (a) -- (e);
      \draw[->,red,thick] (d) -- (b);
      \draw[->,red,thick] (e) -- (c);
      \draw (b) -- (f);

      \node at (1,-0.5) {$-\nabla f=\{(\tau_1,\sigma_2),(\sigma_1,\nu_2)\}$};
    \end{scope}

    \begin{scope}[shift={(4,0)}]
      \node (a) at (0,0) {$1$};
      \node (b) at (0,1.5) {$2$};
      \node (c) at (0,3) {$3$};
      \node (d) at (2,0) {$1$};
      \node (e) at (2,1.5) {$2$};
      \node (f) at (2,3) {$3$};

      \draw (a) -- (b);
      \draw (b) -- (c);
      \draw (d) -- (e);
      \draw (e) -- (f);
      \draw (a) -- (e);
      \draw (d) -- (b);
      \draw (e) -- (c);
      \draw (b) -- (f);

      \node at (1,-0.5) {$-\nabla f'=\emptyset$};
    \end{scope}

\end{tikzpicture}

\caption{Three diagrams illustrating a Lefschetz complex $\K$, a discrete Morse
function $f$ with its gradient vector field $-\nabla f$, and another discrete
Morse function $f'$ on $\K$ for which $-\nabla f'$ is empty.}
\label{fig:gvf-example}
\end{figure}

To state the Morse inequality, we count the cells of $\K$ according to the values of
$U$ and $D$.

\begin{definition}
    Let $f\colon\K\to\R$ be a discrete Morse function. For each $k\in\Z_{\ge0}$, we introduce the following notation:
    \begin{align*}
        n_k(\K)&:=\rank C_k(\K)=\#\{\text{$k$-cells of $\K$}\}\\
        m_k(\K,f)&:=\#\{\text{critical $k$-cells of $\K$}\}=\#\{\sigma^k\in \K\mid U(\sigma)=D(\sigma)=0\}\\
        u_k(\K,f)&:=\#\{\sigma^k\in \K\mid U(\sigma)=1\}\\
        d_k(\K,f)&:=\#\{\sigma^k\in \K\mid D(\sigma)=1\}
    \end{align*}
\end{definition}

\begin{lemma}\label{lem:count}
    Let $f\colon\K\to\R$ be a discrete Morse function. For any $k\in\Z_{\ge0}$,
    \[
    n_k=m_k+d_k+u_k,\qquad u_k=d_{k+1}.
    \]
\end{lemma}

\begin{proof}
    By Lemma~\ref{lem:one}, every $k$-cell $\sigma$ satisfies exactly one of
    $U(\sigma)=D(\sigma)=0$, $U(\sigma)=1$ and $D(\sigma)=1$, which gives the first equality.
    If $U(\sigma)=1$ for a $k$-cell $\sigma$, there is a unique cell $\tau$ with
    $(\tau,\sigma)\in-\nabla f$, and $\tau$ is a $(k+1)$-cell with $D(\tau)=1$. By (V2) for
    $-\nabla f$, which is a discrete vector field by Proposition~\ref{prp:gvfisdvf}, the
    assignment $\sigma\mapsto\tau$ is injective, and it is surjective onto the set of
    $(k+1)$-cells $\tau$ with $D(\tau)=1$. This gives the second equality.
\end{proof}

For a discrete Morse function $f\colon\K\to\R$, we define the formal power series
$r_t(\K,f)$ with integer coefficients by
\[
r_t(\K,f):=\sum_{k=1}^\infty \bigl(\rank B_{k-1}(\K) - d_k(\K,f) \bigr)t^{k-1}.
\]

\begin{lemma}\label{lem:poly}
    Let $f\colon\K\to\R$ be a discrete Morse function. Then
    \[
    \sum_{k=0}^\infty m_k t^k=P_t(\K)+(1+t)r_t(\K,f).
    \]
\end{lemma}

\begin{proof}
    Fix $k\in\Z_{\ge0}$. By Lemma~\ref{lem:rankid},
    \[
    n_k=\rank C_k(\K)
    =\rank B_k(\K)+\rank B_{k-1}(\K)+b_k.
    \]
    Since $n_k=m_k+d_k+u_k$ and $u_k=d_{k+1}$ by Lemma~\ref{lem:count}, we obtain
    \begin{align*}
    m_k t^k
    &=b_k t^k+\bigl(\rank B_k(\K)-d_{k+1}\bigr)t^k\\
    &\quad+\bigl(\rank B_{k-1}(\K)-d_k\bigr)t^k
    \end{align*}
    for every $k\ge0$. At $k=0$, the last term vanishes, since
    $B_{-1}(\K)=\{0\}$. Moreover, as $\dim$ takes values in $\Z_{\ge0}$,
    condition~(L1) forces $d_0=0$. Summing over $k\in\Z_{\ge0}$ and reindexing,
    we obtain
    \[
    \sum_{k=0}^\infty m_k t^k
    =P_t(\K)+r_t(\K,f)+t\,r_t(\K,f)
    =P_t(\K)+(1+t)r_t(\K,f).
    \]
\end{proof}

\section{Discrete Morse-Bott functions}\label{sec:MBfunctions}

In this section, we introduce discrete Morse-Bott functions on Lefschetz complexes,
generalizing discrete Morse functions. The weakly critical sets of such a function
assemble into a Lefschetz complex $\W(\K,f)$, which plays the role that the set of
critical cells plays in the Morse case, and we establish an identity for the
Poincar\'{e} series of $\K$. The nonnegativity of the remainder term is not proved here;
it is obtained in Section~\ref{sec:moves} as Theorem~\ref{thm:MB}, using the Moves
introduced there.

\begin{definition}[Discrete Morse-Bott function]
    A function $f\colon\K\to\R$ is a \textbf{discrete Morse-Bott function} on $\K$ if for any cell $\sigma\in \K$, $U^\snc(\sigma)\le1$ and $D^\snc(\sigma)\le 1$.
\end{definition}

As in the case of discrete Morse functions, $U^\snc(\sigma)$ and $D^\snc(\sigma)$ cannot
both be equal to $1$.

\begin{lemma}\label{lem:MBone}
  For any discrete Morse-Bott function $f\colon\K\to\R$ and for any cell $\sigma\in\K$, we have the following inequality
  \[
  U^\snc(\sigma)+D^\snc(\sigma)\le 1.
  \]
\end{lemma}

\begin{proof}
    Suppose that there is a cell $\sigma\in \K$ with $U^\snc(\sigma)=1$ and $D^\snc(\sigma)=1$. By the definition of $U^\snc(\sigma)$ and $D^\snc(\sigma)$, we have two cells $\tau,\nu\in \K$ with $w(\tau,\sigma)w(\sigma,\nu)\neq0$ and $f(\tau)< f(\sigma)< f(\nu)$. By Lemma~\ref{lem:incidence}, we have $\tilde{\sigma}\in \K$ such that $\sigma\neq\tilde{\sigma}$ and $w(\tau,\tilde\sigma)w(\tilde\sigma,\nu)\neq0$.

    The same argument as in the proof of Lemma~\ref{lem:one}, with $D^\snc(\tau)\le1$ and $U^\snc(\nu)\le1$ in place of $D(\tau)\le1$ and $U(\nu)\le1$, shows that $\tilde\sigma$ is counted neither in $D^\snc(\tau)$ nor in $U^\snc(\nu)$. Since the inequalities in the definitions of $U^\snc$ and $D^\snc$ are strict, negating them gives $f(\tau)\ge f(\tilde\sigma)$ and $f(\tilde\sigma)\ge f(\nu)$. Therefore, we obtain the following inequality
  \[
  f(\tau) \ge f(\tilde\sigma) \ge f(\nu) > f(\sigma) > f(\tau),
  \]
  which is a contradiction.
\end{proof}

\begin{definition}
    Let $f\colon\K\to\R$ be a discrete Morse-Bott function. A cell $\sigma\in \K$ is \textbf{weakly critical} for $f$ if $U^\snc(\sigma)+D^\snc(\sigma)=0$. Equivalently, there is a weakly critical set $C\in\C_f$ such that $\sigma\in C$.
\end{definition}

\begin{proposition}\label{prp:MisMB}
    Let $f\colon\K\to\R$ be a discrete Morse function. Then $f$ is a discrete Morse-Bott function, and every critical cell of $f$ is weakly critical.
\end{proposition}

\begin{proof}
    For any cell $\sigma\in \K$, we have $U^\snc(\sigma)\le U(\sigma)\le1$ and $D^\snc(\sigma)\le D(\sigma)\le1$, so $f$ is a discrete Morse-Bott function. If $\sigma$ is critical for $f$, then $U^\snc(\sigma)\le U(\sigma)=0$ and $D^\snc(\sigma)\le D(\sigma)=0$, so $\sigma$ is weakly critical.
\end{proof}

Since $U^\snc$ and $D^\snc$ are defined by strict inequalities, in the Morse-Bott setting
it is the strict gradient vector field that is a discrete vector field.

\begin{proposition}\label{prp:gvfisdvfMB}
    Let $f\colon\K\to\R$ be a discrete Morse-Bott function. The set $-\nabla_s f$ is a discrete vector field on $\K$.
\end{proposition}

\begin{proof}
    Replacing $U$, $D$, $-\nabla f$ and Lemma~\ref{lem:one} by $U^\snc$, $D^\snc$,
    $-\nabla_s f$ and Lemma~\ref{lem:MBone} throughout, the proof of
    Proposition~\ref{prp:gvfisdvf} applies verbatim.
\end{proof}

If $f$ is a discrete Morse-Bott function and $(\tau,\sigma)\in -\nabla_s f$, then we write $\sigma\to_f\tau$.

In the Morse-Bott setting, the role of a single critical cell is played by a whole weakly
critical set. We first show that each such set is a subcomplex, so that their disjoint
union is again a Lefschetz complex.

\begin{proposition}\label{prp:wcsubcpx}
    Let $f\colon\K\to\R$ be a discrete Morse-Bott function. A weakly critical set is a subcomplex of $\K$.
\end{proposition}

\begin{proof}
    Let $C$ be a weakly critical set and $\tau,\nu\in C$. By definition, $f$ is constant on
    $C$ and $U^\snc=D^\snc=0$ on $C$. By Proposition~\ref{prp:subcpx}
    it suffices to verify (L2) for $C$, and since (L2) holds for $\K$, it is enough to show
    $w(\tau,\sigma)w(\sigma,\nu)=0$ for every $\sigma\in\K\setminus C$.

    Suppose $w(\tau,\sigma)w(\sigma,\nu)\neq0$ for some $\sigma\in\K\setminus C$. Since
    $D^\snc(\tau)=0$ and $U^\snc(\nu)=0$, we have $f(\nu)\le f(\sigma)\le f(\tau)=f(\nu)$, so
    $\sigma$ lies on the same level set as $C$. As $\sigma\notin C$, the cell $\sigma$ is
    not weakly critical, so $\beta\to_f\sigma$ or $\sigma\to_f\beta$ for some $\beta\in\K$.

    Assume $\beta\to_f\sigma$, that is, $w(\sigma,\beta)\neq0$ and $f(\sigma)<f(\beta)$. By
    Lemma~\ref{lem:incidence} there is a cell $\widetilde\sigma\neq\sigma$ with
    $w(\tau,\widetilde\sigma)w(\widetilde\sigma,\beta)\neq0$. The cell $\sigma$ is counted
    in $U^\snc(\beta)$, hence, since $U^\snc(\beta)\le1$, it is the only cell counted in it,
    and therefore $f(\beta)\le f(\widetilde\sigma)$. Since $D^\snc(\tau)=0$, we also have
    $f(\widetilde\sigma)\le f(\tau)$. Hence
    \[
    f(\nu)\le f(\sigma)<f(\beta)\le f(\widetilde{\sigma})\le f(\tau)=f(\nu),
    \]
    which is a contradiction. The case $\sigma\to_f\beta$ is proved in the same way, using
    $D^\snc(\beta)\le1$ and $U^\snc(\nu)=0$.
\end{proof}

For a discrete Morse-Bott function $f\colon\K\to\R$, each $C\in\C_f$ is a subcomplex of
$\K$ by Proposition~\ref{prp:wcsubcpx}, and every cell of such a $C$ is a cell of $\K$, so
the finiteness assumption of Lemma~\ref{lem:disjoint} is satisfied and the disjoint union
below is a Lefschetz complex.

\begin{definition}\label{def:wccpx}
    Let $f\colon\K\to\R$ be a discrete Morse-Bott function. The \textbf{weakly critical
    complex} of $(\K,f)$ is the Lefschetz complex
    \[
    \W(\K,f):=\bigsqcup_{C\in\C_f}\bigl(C,\dim|_C,w|_{C\times C}\bigr),
    \]
    and we write $\W$ when $\K$ and $f$ are clear from the context. Its cells are the
    weakly critical cells of $f$, and by Remark~\ref{rem:wcunique} its incidence function
    $w^\W$ is given by
    \[
    w^\W(\tau,\sigma)=
    \begin{cases}
        w(\tau,\sigma) & \text{if } f(\tau)=f(\sigma),\\
        0 & \text{otherwise}.
    \end{cases}
    \]
\end{definition}

To state the Morse-Bott inequality, we count the cells of $\K$ according to the values of
$U^\snc$ and $D^\snc$.

\begin{definition}
    Let $f\colon\K\to\R$ be a discrete Morse-Bott function. For each $k\in\Z_{\ge0}$, we introduce the following notation:
    \begin{align*}
        m_k^\snc(\K,f)
        &:=\#\{\text{weakly critical $k$-cells of $\K$}\}\\
        &=\#\{\sigma^k\in \K\mid U^\snc(\sigma)=D^\snc(\sigma)=0\}\\
        u_k^\snc(\K,f)&:=\#\{\sigma^k\in \K\mid U^\snc(\sigma)=1\}\\
        d_k^\snc(\K,f)&:=\#\{\sigma^k\in \K\mid D^\snc(\sigma)=1\}
    \end{align*}
\end{definition}

\begin{remark}\label{rem:wccpx}
    By Definition~\ref{def:wccpx}, $\rank C_k(\W)=m_k^\snc$ for every $k$, and
    Lemma~\ref{lem:disjointsum} gives
    \[
    \rank B_k(\W)=\sum_{C\in\C_f}\rank B_k(C),\qquad
    P_t(\W)=\sum_{C\in\C_f}P_t(C).
    \]
\end{remark}

\begin{lemma}\label{lem:countMB}
    Let $f\colon\K\to\R$ be a discrete Morse-Bott function. For any $k\in\Z_{\ge0}$,
    \[
    n_k=m_k^\snc+d_k^\snc+u_k^\snc,\qquad u_k^\snc=d_{k+1}^\snc.
    \]
\end{lemma}

\begin{proof}
    The proof is the same as that of Lemma~\ref{lem:count}, using $U^\snc$, $D^\snc$ and
    $-\nabla_s f$ in place of $U$, $D$ and $-\nabla f$, and citing Lemma~\ref{lem:MBone}
    and Proposition~\ref{prp:gvfisdvfMB} in place of Lemma~\ref{lem:one} and
    Proposition~\ref{prp:gvfisdvf}.
\end{proof}

For a discrete Morse-Bott function $f\colon\K\to\R$, we define the formal power series
$R_t(\K,f)$ with integer coefficients by
\[
R_t(\K,f):=\sum_{k=1}^\infty \Bigl(\rank B_{k-1}(\K) - d_k^\snc(\K,f) - \rank B_{k-1}(\W(\K,f))\Bigr)t^{k-1}.
\]
For $k\in\Z_{\ge0}$, we write $R_t(\K,f)_k$ for the coefficient of $t^k$ in $R_t(\K,f)$.

\begin{lemma}\label{lem:MBpoly}
    Let $f\colon\K\to\R$ be a discrete Morse-Bott function. Then
    \[
    \sum_{C\in\C_f}P_t(C)=P_t(\K)+(1+t)R_t(\K,f).
    \]
\end{lemma}

\begin{proof}
    Fix $k\in\Z_{\ge0}$. By Lemmas~\ref{lem:rankid} and~\ref{lem:countMB}, we obtain
    \[
    m_k^\snc=b_k^\K+\bigl(\rank B_k(\K)-d_{k+1}^\snc\bigr)+\bigl(\rank B_{k-1}(\K)-d_k^\snc\bigr).
    \]
    Applying Lemma~\ref{lem:rankid} with $\W(\K,f)$ in place of $\K$ and using
    $\rank C_k(\W)=m_k^\snc$
    (Remark~\ref{rem:wccpx}), we obtain
    \[
    m_k^\snc=\rank B_k(\W)+\rank B_{k-1}(\W)+b_k^{\W}.
    \]
    Subtracting the second equality from the first, we obtain
    \begin{align*}
    b_k^{\W}={}&b_k^\K+\bigl(\rank B_k(\K)-d_{k+1}^\snc-\rank B_k(\W)\bigr)\\
    &+\bigl(\rank B_{k-1}(\K)-d_k^\snc-\rank B_{k-1}(\W)\bigr)
    \end{align*}
    for every $k\ge0$. At $k=0$, the last term vanishes, since
    $B_{-1}(\K)=\{0\}=B_{-1}(\W)$ and, as in the proof of Lemma~\ref{lem:poly},
    $d_0^\snc=0$. Summing over $k\in\Z_{\ge0}$ and
    reindexing, we obtain
    \[
    P_t(\W)=P_t(\K)+R_t(\K,f)+t\,R_t(\K,f)=P_t(\K)+(1+t)R_t(\K,f).
    \]
    Since $P_t(\W)=\sum_{C\in\C_f}P_t(C)$ by Remark~\ref{rem:wccpx}, this is the desired
    equality.
\end{proof}

We now compare the Morse and the Morse-Bott settings. For a discrete Morse function, the
weakly critical complex is the disjoint union of the critical cells of $f$ and of pairs of
cells with nonzero incidence number. This will identify the remainder terms $r_t(\K,f)$
and $R_t(\K,f)$.

\begin{lemma}\label{lem:samevaluewc}
    Let $f\colon\K\to\R$ be a discrete Morse function and let $\tau,\sigma\in\K$ satisfy
    $w(\tau,\sigma)\neq0$ and $f(\tau)=f(\sigma)$. Then $\tau$ and $\sigma$ are weakly
    critical for $f$.
\end{lemma}

\begin{proof}
    Since $w(\tau,\sigma)\neq0$ and $f(\tau)\le f(\sigma)$, the cell $\sigma$ is counted in
    $D(\tau)$, so $D(\tau)=1$ with $\sigma$ as the only such cell, and $U(\tau)=0$ by
    Lemma~\ref{lem:one}. If a cell $\nu$ were counted in $D^\snc(\tau)$, it would also be
    counted in $D(\tau)$, hence $\nu=\sigma$; but $f(\tau)<f(\sigma)$ fails. Thus
    $D^\snc(\tau)=0$, and $U^\snc(\tau)\le U(\tau)=0$, so $\tau$ is weakly critical. The
    assertion for $\sigma$ follows in the same way from $f(\sigma)\ge f(\tau)$, using
    $U(\sigma)=1$ and $D(\sigma)=0$.
\end{proof}

Thus the nonzero incidences of $\W(\K,f)$ are exactly the pairs $(\tau,\sigma)$ with
$w(\tau,\sigma)\neq0$ and $f(\tau)=f(\sigma)$; this is what makes the following
decomposition possible.

\begin{proposition}\label{prp:wcofmorse}
    Let $f\colon\K\to\R$ be a discrete Morse function. Then $\W(\K,f)$ is the disjoint union
    of the one-cell complexes given by the critical cells of $f$ and of the two-cell
    complexes $\{\tau,\sigma\}$ with $w^\W(\tau,\sigma)\neq0$. Consequently
    \[
    \sum_{C\in\C_f}P_t(C)=P_t(\W(\K,f))=\sum_{k=0}^\infty m_k t^k.
    \]
\end{proposition}

\begin{proof}
    By Proposition~\ref{prp:MisMB}, $f$ is a discrete Morse-Bott function, so
    $\W:=\W(\K,f)$ is defined. By Definition~\ref{def:wccpx} and
    Lemma~\ref{lem:samevaluewc}, $w^\W(\tau,\sigma)\neq0$ if and only if $w(\tau,\sigma)\neq0$
    and $f(\tau)=f(\sigma)$; any such pair lies in $-\nabla f$, so by
    Proposition~\ref{prp:gvfisdvf} distinct such pairs have disjoint underlying sets. Hence
    $\W$ is the disjoint union of the two-cell complexes $\{\tau,\sigma\}$ with
    $w^\W(\tau,\sigma)\neq0$ and of the remaining cells, regarded as one-cell complexes.

    A cell of $\W$ lies in no such pair if and only if it is critical for $f$. If $\sigma$ is
    critical, then $U(\sigma)=D(\sigma)=0$, so $w(\tau,\sigma)\neq0$ gives $f(\tau)>f(\sigma)$
    and $w(\sigma,\nu)\neq0$ gives $f(\nu)<f(\sigma)$; hence no cell incident to $\sigma$ has
    the same $f$-value. If $\sigma$ is not critical, then $U(\sigma)+D(\sigma)=1$ by
    Lemma~\ref{lem:one}; assuming $U(\sigma)=1$ and letting $\tau$ be the cell counted in it,
    we have $f(\tau)\le f(\sigma)$, while $U^\snc(\sigma)=0$ gives $f(\tau)\ge f(\sigma)$, so
    $f(\tau)=f(\sigma)$ and $\sigma$ lies in the pair $(\tau,\sigma)$. The case $D(\sigma)=1$
    is symmetric.

    For a two-cell summand the boundary homomorphism is an isomorphism, so its
    Poincar\'{e} series vanishes, while a critical $k$-cell contributes $t^k$. The first
    equality is Remark~\ref{rem:wccpx}, and the second follows from
    Lemma~\ref{lem:disjointsum} applied to the decomposition just established.
\end{proof}

\begin{remark}\label{rem:MtoMB}
    Let $f$ be a discrete Morse function on $\K$. By Proposition~\ref{prp:MisMB}, $f$ is a
    discrete Morse-Bott function, so Lemma~\ref{lem:MBpoly} applies to it. By
    Proposition~\ref{prp:wcofmorse}, we have
    $\sum_{C\in\C_f}P_t(C)=\sum_{k=0}^\infty m_k t^k$. Comparing
    Lemmas~\ref{lem:poly} and~\ref{lem:MBpoly} and dividing by $1+t$, we obtain
    $R_t(\K,f)=r_t(\K,f)$.
\end{remark}

The dimension function is the basic example of a discrete Morse-Bott function. It will be
used in Section~\ref{sec:moves} and in the proof of the nonnegativity of $R_t(\K,f)$.

\begin{lemma}\label{lem:dimMB}
    The function $\dim$ is a discrete Morse-Bott function on $\K$, every cell of $\K$ is
    weakly critical for $\dim$, and the incidence function of $\W(\K,\dim)$ vanishes
    identically. In particular, for every $k$ we have $d^\snc_k(\K,\dim)=0$ and
    $\rank B_k(\W(\K,\dim))=0$.
\end{lemma}

\begin{proof}
    If $w(\tau,\sigma)\neq0$, then $\dim\tau=\dim\sigma+1$ by (L1), so the inequality
    $\dim\sigma>\dim\tau$ fails; hence $\tau$ is not counted in $U^\snc(\sigma)$ and
    $\sigma$ is not counted in $D^\snc(\tau)$. Therefore
    $U^\snc(\sigma)=D^\snc(\sigma)=0$ for every cell $\sigma$, and in particular
    $d^\snc_k(\K,\dim)=0$ for every $k$. Moreover $\dim\tau=\dim\sigma$ never holds for a
    pair with $w(\tau,\sigma)\neq0$, so $w^\W$ vanishes identically by
    Definition~\ref{def:wccpx}. Hence the boundary homomorphism of $\W(\K,\dim)$ is zero
    and $\rank B_k(\W(\K,\dim))=0$ for every $k$.
\end{proof}

\section{Moves}\label{sec:moves}

Throughout this section, unless otherwise stated, let $f\colon\K\to\R$ be a discrete
Morse-Bott function. We modify the pair $(\K,f)$ while keeping $R_t(\K,f)$ unchanged.
To this end, we introduce four \emph{Moves}. Move~(I) cancels a pair of cells with
nonzero incidence number, provided that the pair belongs to the discrete vector field
$-\nabla_s f$ or that the two cells lie in a common weakly critical set.
Move~(II) removes an isolated cell. Move~(III) replaces $f$ by a function which orders
in the same way any two cells connected by a nonzero incidence number. Move~(IV)
cancels a pair of cells with nonzero incidence number and adjoins the elementary block
determined by that pair, under the assumption that $f$ agrees with $\dim$ in the two
dimensions concerned. Thus Moves~(I), (II) and~(IV) modify $\K$, while Move~(III)
modifies $f$.

By Remark~\ref{rem:Ptcriterion} below, it suffices to track the two Poincar\'{e} series
$P_t(\K)$ and $P_t(\W(\K,f))$. The effect of each Move on them is as follows, where $d$
denotes the dimension of the cell removed by Move~(II).
\[
\begin{array}{lcc}
\text{Move} & \text{change in } P_t(\K) & \text{change in } P_t(\W(\K,f))\\\hline
\text{(I)} & \text{none} & \text{none}\\
\text{(II)} & -\,t^{d} & -\,t^{d}\\
\text{(III)} & \text{none} & \text{none}\\
\text{(IV)} & \text{none} & \text{none}
\end{array}
\]
Move~(II) is thus the only Move that changes $P_t(\K)$, and Move~(IV) is the only Move
that adjoins a new summand.

The conclusion of each Move below is an equality between the series $R_t$ of the old pair and
of the new one, and the Moves are applied one after another below. We therefore introduce a
notation for this relation and record that it is an equivalence relation.

\begin{definition}\label{def:sim}
    A \textbf{Morse-Bott pair} is a pair consisting of a Lefschetz complex and a discrete
    Morse-Bott function on it. For Morse-Bott pairs $(\K,f)$ and $(\K',f')$, we write
    $(\K,f)\sim(\K',f')$ if $R_t(\K,f)=R_t(\K',f')$. Being defined by an equality of formal
    power series, $\sim$ is an equivalence relation on Morse-Bott pairs.
\end{definition}

\begin{remark}\label{rem:Ptcriterion}
    Let $(\K,f)$ and $(\K',f')$ be Morse-Bott pairs. If
    \[
    P_t\bigl(\W(\K,f)\bigr)-P_t(\K)=P_t\bigl(\W(\K',f')\bigr)-P_t(\K'),
    \]
    then $(\K,f)\sim(\K',f')$. Indeed, by Lemma~\ref{lem:MBpoly} and Remark~\ref{rem:wccpx},
    applied to $(\K,f)$ and to $(\K',f')$, the hypothesis reads
    $(1+t)R_t(\K,f)=(1+t)R_t(\K',f')$, and hence $R_t(\K,f)=R_t(\K',f')$. In the verification of the Moves below we use this criterion
    rather than the defining formula of $R_t$.
\end{remark}

Move~(IV) below adjoins a new summand, and in the proof of
Theorem~\ref{thm:elementary} the pair $(\K,f)$ is compared with a disjoint union of
elementary blocks. We therefore record how the weakly critical complex and $R_t$ behave
under disjoint unions. Given functions $g\colon\K_1\to\R$ and $h\colon\K_2\to\R$ on
Lefschetz complexes $\K_1$ and $\K_2$, we write $g\sqcup h$ for the function on
$\K_1\sqcup\K_2$ which restricts to $g$ on $\K_1$ and to $h$ on $\K_2$.

\begin{lemma}\label{lem:Rtdisjoint}
    Let $\{\K_i\}_{i\in I}$ be a family of Lefschetz complexes whose disjoint union
    $\bigsqcup_{i\in I}\K_i$ is a Lefschetz complex, and let
    $f\colon\bigsqcup_{i\in I}\K_i\to\R$ be a discrete Morse-Bott function on $\K_i$
    for each $i\in I$. Then $f$ is a discrete Morse-Bott function on
    $\bigsqcup_{i\in I}\K_i$,
    \[
    \W\Bigl(\bigsqcup_{i\in I}\K_i,f\Bigr)=\bigsqcup_{i\in I}\W(\K_i,f),
    \]
    and $R_t\bigl(\bigsqcup_{i\in I}\K_i,f\bigr)=\sum_{i\in I}R_t(\K_i,f)$.
\end{lemma}

\begin{proof}
    Write $\L:=\bigsqcup_{i\in I}\K_i$ and let $w^\L$ be its incidence function. Cells of
    $\L$ lying in distinct summands have incidence number zero, so
    $U^\snc_{\L,f}(\sigma)=U^\snc_{\K_i,f}(\sigma)$ and
    $D^\snc_{\L,f}(\sigma)=D^\snc_{\K_i,f}(\sigma)$ for $\sigma\in\K_i$. In particular $f$
    is a discrete Morse-Bott function on $\L$, and a cell of $\K_i$ is weakly critical for
    $f$ on $\L$ if and only if it is weakly critical for $f$ on $\K_i$. Hence $\W(\L,f)$
    and $\bigsqcup_{i\in I}\W(\K_i,f)$ have the same cells. Their incidence functions
    agree as well. Indeed, by Definition~\ref{def:wccpx}, the incidence number of two
    cells $\tau,\sigma$ of $\W(\L,f)$ is $w^\L(\tau,\sigma)$ if $f(\tau)=f(\sigma)$ and
    $0$ otherwise. If $\tau$ and $\sigma$ lie in distinct summands, this number is zero,
    and so is the incidence number of $\tau$ and $\sigma$ in
    $\bigsqcup_{i\in I}\W(\K_i,f)$. If both lie in $\K_i$, then $w^\L(\tau,\sigma)$ is
    their incidence number in $\K_i$, and Definition~\ref{def:wccpx} applied to $\K_i$
    gives the same value in $\W(\K_i,f)$. This proves the first two assertions.

    Let $k\in\Z_{\ge0}$. If $\K_i$ has no $(k+1)$-cell, then
    \[
    \rank B_k(\K_i)=d^\snc_{k+1}(\K_i,f)
    =\rank B_k\bigl(\W(\K_i,f)\bigr)=0.
    \]
    Thus, the coefficient $R_t(\K_i,f)_k$ of $t^k$ is zero. By
    Lemma~\ref{lem:disjointsum},
    only finitely many $i\in I$ satisfy $\rank C_{k+1}(\K_i)\neq0$, and therefore
    $R_t(\K_i,f)_k=0$ for all but finitely many $i\in I$. Hence
    $\sum_{i\in I}R_t(\K_i,f)$ is a well-defined formal power series, and the sums below
    may be rearranged coefficientwise. Apply Lemma~\ref{lem:MBpoly} and
    Remark~\ref{rem:wccpx} first to $\L$ and then to each $\K_i$. We obtain
    \begin{align*}
    (1+t)R_t(\L,f)&=P_t\bigl(\W(\L,f)\bigr)-P_t(\L)\\
    &=\sum_{i\in I}P_t\bigl(\W(\K_i,f)\bigr)-\sum_{i\in I}P_t(\K_i)\\
    &=\sum_{i\in I}\Bigl(P_t\bigl(\W(\K_i,f)\bigr)-P_t(\K_i)\Bigr)\\
    &=(1+t)\sum_{i\in I}R_t(\K_i,f),
    \end{align*}
    where the second equality applies Lemma~\ref{lem:disjointsum} to $\L$ and to
    $\W(\L,f)=\bigsqcup_{i\in I}\W(\K_i,f)$. Hence
    $R_t(\L,f)=\sum_{i\in I}R_t(\K_i,f)$.
\end{proof}

Moves~(I) and~(IV) both cancel a pair of cells with nonzero incidence number. We now
define the triple obtained by removing such a pair and correcting the incidence numbers
of the remaining cells.

\begin{definition}\label{def:quotient}
    Let $\mu,\kappa\in \K$ with $w(\mu,\kappa)\neq0$. We define the triple
    $\K\{\mu,\kappa\}:=(X',\dim',w')$, where $X':=X\setminus\{\mu,\kappa\}$, the function
    $\dim'\colon X'\to\Z_{\ge0}$ is given by $\dim':=\dim|_{X'}$, and the function
    $w'\colon X'\times X'\to\R$ is given by
    \[
    w'(\tau,\sigma):=w(\tau,\sigma)-\frac{w(\tau,\kappa)w(\mu,\sigma)}{w(\mu,\kappa)}.
    \]
\end{definition}

\begin{remark}
    The triple $\K\{\mu,\kappa\}$ coincides with the \emph{quotient} of $\K$ by the vector $(\kappa,\mu)$ introduced in \cite[Def.~3.9]{edel2024DMFLef}, where the passage from $\K$ to $\K\{\mu,\kappa\}$ is called the \emph{cancellation} of that vector. Remark~\ref{rem:quotientinc} below corresponds to \cite[Prop.~3.12\,(ii)]{edel2024DMFLef}. Proposition~\ref{prp:quotientLef} below, that $\K\{\mu,\kappa\}$ is again a Lefschetz complex, corresponds to \cite[Prop.~3.10]{edel2024DMFLef}. Lemma~\ref{lem:ChHoMuKappa} below, that $C_\bullet(\K)$ and $C_\bullet(\K\{\mu,\kappa\})$ are chain homotopy equivalent, corresponds to \cite[Thm.~3.11]{edel2024DMFLef}. We give an independent proof of Proposition~\ref{prp:quotientLef}, since it underlies all of the constructions below. For Lemma~\ref{lem:ChHoMuKappa}, we explain why the construction in \cite[Lem.~2.5]{mischaikow2013morse} applies in the present setting.
\end{remark}

\begin{remark}\label{rem:quotientinc}
    Let $\tau,\sigma\in\K\{\mu,\kappa\}$ and denote by $w'$ the incidence function of
    $\K\{\mu,\kappa\}$. By Definition~\ref{def:quotient}, if $w(\tau,\kappa)=0$ or
    $w(\mu,\sigma)=0$, then $w'(\tau,\sigma)=w(\tau,\sigma)$.
\end{remark}

\begin{example}
Figure~\ref{fig:quotient} illustrates the construction of triples described in
Definition~\ref{def:quotient}. Let $\K$ be the Lefschetz complex shown in the
diagram on the left, whose cells are $\nu_i$, $\sigma_i$, $\tau_i$ and $\alpha_i$
for $i\in\{1,2\}$. The incidence numbers $w(\tau_1,\sigma_1)$
and $w(\sigma_1,\nu_1)$ are both equal to $1$, so Definition~\ref{def:quotient}
applies to the pairs $(\tau_1,\sigma_1)$ and $(\sigma_1,\nu_1)$. The center
diagram shows the triple $\K\{\tau_1,\sigma_1\}$, obtained by removing the cells
$\tau_1$ and $\sigma_1$ and correcting the incidence numbers of the remaining
cells. This correction makes the incidence number of the pair $(\tau_2,\sigma_2)$
zero. The diagram on the right shows the triple $\K\{\sigma_1,\nu_1\}$, obtained
in the same way from the pair $(\sigma_1,\nu_1)$. There the correction instead
makes the incidence number of the pair $(\sigma_2,\nu_2)$ zero. In both cases the correction changes an incidence number between two cells that
are not removed. Neither triple is obtained from $\K$ by discarding two cells and
keeping the remaining incidence numbers unchanged. Which incidence number changes
depends on the pair that is removed, and Remark~\ref{rem:quotientinc} identifies
it in each case.
\end{example}

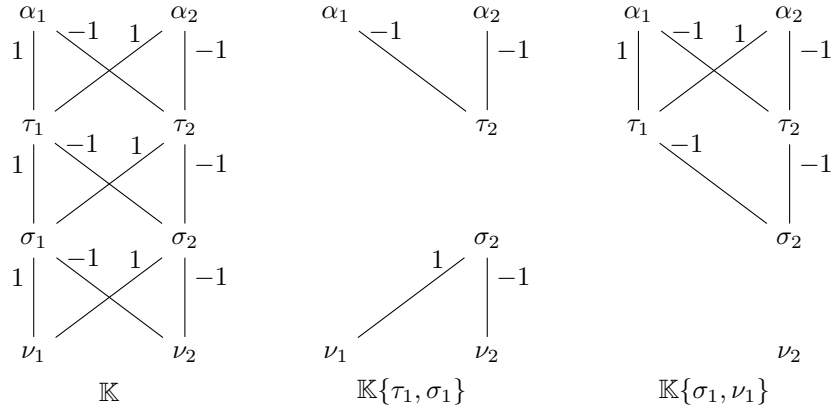
\begin{figure}[htbp]
\centering
\begin{tikzpicture}[scale=1]

    \begin{scope}[shift={(-4,0)}]
      \node (a) at (0,0) {$\nu_1$};
      \node (b) at (0,1.5) {$\sigma_1$};
      \node (c) at (0,3) {$\tau_1$};
      \node (d) at (2,0) {$\nu_2$};
      \node (e) at (2,1.5) {$\sigma_2$};
      \node (f) at (2,3) {$\tau_2$};

      \node (m) at (0,4.5) {$\alpha_1$};
      \node (n) at (2,4.5) {$\alpha_2$};

      \draw (a) -- (b) node[near end, left] {$1$};
      \draw (b) -- (c) node[near end, left] {$1$};
      \draw (d) -- (e) node[near end, right] {$-1$};
      \draw (e) -- (f) node[near end, right] {$-1$};
      \draw (a) -- (e) node[near end, above] {$1$};
      \draw (d) -- (b) node[near end, above] {$-1$};
      \draw (e) -- (c) node[near end, above] {$-1$};
      \draw (b) -- (f) node[near end, above] {$1$};

      \draw (c) -- (m) node[near end, left] {$1$};
      \draw (f) -- (n) node[near end, right] {$-1$};
      \draw (f) -- (m) node[near end, above] {$-1$};
      \draw (c) -- (n) node[near end, above] {$1$};

      \node at (1,-0.5) {$\K$};
    \end{scope}

    \begin{scope}[shift={(0,0)}]
      \node (a) at (0,0) {$\nu_1$};
      \node (d) at (2,0) {$\nu_2$};
      \node (e) at (2,1.5) {$\sigma_2$};
      \node (f) at (2,3) {$\tau_2$};

      \node (m) at (0,4.5) {$\alpha_1$};
      \node (n) at (2,4.5) {$\alpha_2$};

      \draw (d) -- (e) node[near end, right] {$-1$};
      \draw (a) -- (e) node[near end, above] {$1$};

      \draw (f) -- (n) node[near end, right] {$-1$};
      \draw (f) -- (m) node[near end, above] {$-1$};

      \node at (1,-0.5) {$\K\{\tau_1,\sigma_1\}$};
    \end{scope}

    \begin{scope}[shift={(4,0)}]
      \node (c) at (0,3) {$\tau_1$};
      \node (d) at (2,0) {$\nu_2$};
      \node (e) at (2,1.5) {$\sigma_2$};
      \node (f) at (2,3) {$\tau_2$};

      \node (m) at (0,4.5) {$\alpha_1$};
      \node (n) at (2,4.5) {$\alpha_2$};

      \draw (e) -- (f) node[near end, right] {$-1$};
      \draw (e) -- (c) node[near end, above] {$-1$};

      \draw (c) -- (m) node[near end, left] {$1$};
      \draw (f) -- (n) node[near end, right] {$-1$};
      \draw (f) -- (m) node[near end, above] {$-1$};
      \draw (c) -- (n) node[near end, above] {$1$};

      \node at (1,-0.5) {$\K\{\sigma_1,\nu_1\}$};
    \end{scope}

\end{tikzpicture}

\caption{Three diagrams illustrating a Lefschetz complex $\K$, the triple
$\K\{\tau_1,\sigma_1\}$, and the triple $\K\{\sigma_1,\nu_1\}$.}
\label{fig:quotient}
\end{figure}

\begin{proposition}\label{prp:quotientLef}
    The triple $\K\{\mu,\kappa\}=(X',\dim',w')$ is a Lefschetz complex.
\end{proposition}

\begin{proof}
    For every $k\in\Z_{\ge0}$, we have
    \[
    (\dim')^{-1}(k)\subset\dim^{-1}(k).
    \]
    Hence $(\dim')^{-1}(k)$ is finite. It remains to verify conditions (L1) and (L2)
    for $\K\{\mu,\kappa\}$.

    For (L1), let $\tau,\sigma\in X'$ with $w'(\tau,\sigma)\neq0$.
    Suppose first that $w(\tau,\sigma)\neq0$. Then
    \[
    \dim'\tau-\dim'\sigma=\dim\tau-\dim\sigma=1,
    \]
    by (L1) for $\K$. Otherwise, $w(\tau,\sigma)=0$, and hence
    \[
    \frac{w(\tau,\kappa)w(\mu,\sigma)}{w(\mu,\kappa)}\neq0.
    \]
    It follows that $w(\tau,\kappa)\neq0$ and $w(\mu,\sigma)\neq0$. By (L1) for $\K$,
    we have $\dim\tau-\dim\kappa=1$,
    $\dim\mu-\dim\sigma=1$ and $\dim\mu-\dim\kappa=1$, so
    \[
    \dim'\tau-\dim'\sigma=\dim\tau-\dim\sigma=(\dim\kappa+1)-(\dim\mu-1)=\dim\kappa+1-\dim\kappa=1.
    \]
    Thus, (L1) holds for $\K\{\mu,\kappa\}$.
    
    For (L2), let $\tau,\nu\in X'$. For $\alpha,\beta\in X$, put
    $S(\alpha,\beta):=\sum_{\sigma\in X'}w(\alpha,\sigma)w(\sigma,\beta)$. For
    $\sigma\in X'$, by definition,
    \begin{align*}
    w'(\tau,\sigma)
    &=w(\tau,\sigma)
      -\frac{w(\tau,\kappa)w(\mu,\sigma)}{w(\mu,\kappa)},\\
    w'(\sigma,\nu)
    &=w(\sigma,\nu)
      -\frac{w(\sigma,\kappa)w(\mu,\nu)}{w(\mu,\kappa)}.
    \end{align*}
    Multiplying the displayed formulas and summing over $\sigma\in X'$, we obtain
    \begin{align*}
    \sum_{\sigma\in X'}w'(\tau,\sigma)w'(\sigma,\nu)
    &=S(\tau,\nu)
      -\frac{w(\mu,\nu)}{w(\mu,\kappa)}S(\tau,\kappa)\\
    &\quad-\frac{w(\tau,\kappa)}{w(\mu,\kappa)}S(\mu,\nu)
      +\frac{w(\tau,\kappa)w(\mu,\nu)}{w(\mu,\kappa)^2}S(\mu,\kappa).
    \end{align*}
    Since $X=X'\sqcup\{\mu,\kappa\}$, condition (L2) for $\K$ gives
    \[
    S(\alpha,\beta)=-w(\alpha,\mu)w(\mu,\beta)-w(\alpha,\kappa)w(\kappa,\beta)
    \]
    for every $\alpha,\beta\in X$. By (L1) for $\K$, we have $w(\mu,\mu)=w(\kappa,\kappa)=0$,
    so $S(\tau,\kappa)=-w(\tau,\mu)w(\mu,\kappa)$, $S(\mu,\nu)=-w(\mu,\kappa)w(\kappa,\nu)$
    and $S(\mu,\kappa)=0$. Substituting these, the second and third terms become
    $w(\tau,\mu)w(\mu,\nu)$ and $w(\tau,\kappa)w(\kappa,\nu)$, which cancel the corresponding
    terms of $S(\tau,\nu)$, and the fourth term vanishes. Thus, (L2) holds for
    $\K\{\mu,\kappa\}$.
\end{proof}

\begin{lemma}\label{lem:ChHoMuKappa}
    The chain complexes $C_\bullet(\K)$ and $C_\bullet(\K\{\mu,\kappa\})$ are chain
    homotopy equivalent.
\end{lemma}

\begin{proof}
    The chain maps and chain homotopy used in the proof of
    \cite[Lem.~2.5]{mischaikow2013morse} remain well-defined in the present setting,
    since there are only finitely many cells in each degree. The same identities
    therefore establish the chain homotopy equivalence.
\end{proof}

\begin{lemma}\label{lem:rankcancel}
    For every $k\in\Z_{\ge0}$,
    \[
    \rank B_k(\K)=
    \begin{cases}
        \rank B_k(\K\{\mu,\kappa\})+1 & \text{if } k=\dim\kappa,\\
        \rank B_k(\K\{\mu,\kappa\}) & \text{if } k\neq\dim\kappa.
    \end{cases}
    \]
\end{lemma}

\begin{proof}
    Write $\K':=\K\{\mu,\kappa\}$. For every $k\in\Z_{\ge-1}$, put
    $\gamma_k:=\rank C_k(\K)-\rank C_k(\K')$ and
    $\delta_k:=\rank B_k(\K)-\rank B_k(\K')$. By Lemmas~\ref{lem:ChHoMuKappa}
    and~\ref{lem:rankhom} we have $b_k^\K=b_k^{\K'}$, so Lemma~\ref{lem:rankid} applied
    to $\K$ and to $\K'$ yields $\gamma_k=\delta_k+\delta_{k-1}$ for every $k\ge0$.
    Since $\gamma_k=1$ for $k\in\{\dim\mu,\dim\kappa\}$ and $\gamma_k=0$ otherwise, and
    $\delta_{-1}=0$, induction on $k$ gives the assertion.
\end{proof}

\begin{lemma}\label{lem:cancelvalue}
    Let $\mu,\kappa\in\K$ with $w(\mu,\kappa)\neq0$. Assume that either
    \begin{enumerate}[(a)]
        \item $\kappa\to_f\mu$, or
        \item $\mu$ and $\kappa$ are weakly critical and $f(\mu)=f(\kappa)$.
    \end{enumerate}
    Then for any cells $\alpha,\beta\in\K\{\mu,\kappa\}$,
    \begin{align*}
        w(\mu,\alpha)\neq0&\implies f(\alpha)\le f(\mu), &
        w(\beta,\mu)\neq0&\implies f(\mu)\le f(\beta),\\
        w(\kappa,\alpha)\neq0&\implies f(\alpha)\le f(\kappa), &
        w(\beta,\kappa)\neq0&\implies f(\kappa)\le f(\beta).
    \end{align*}
\end{lemma}

\begin{proof}
    In case (a), since $f(\mu)<f(\kappa)$, the cell $\kappa$ is counted in
    $D^\snc(\mu)$, so $D^\snc(\mu)=1$ and $\kappa$ is the only cell counted in it.
    Likewise $U^\snc(\kappa)=1$ and $\mu$ is the only cell counted in it. By
    Lemma~\ref{lem:MBone}, $U^\snc(\mu)=0$ and $D^\snc(\kappa)=0$. Let
    $\alpha,\beta\in\K\{\mu,\kappa\}$. Then $\alpha\neq\kappa$ and $\beta\neq\mu$.
    If $w(\mu,\alpha)\neq0$, then $\alpha$ is not counted in $D^\snc(\mu)$, that is,
    $f(\mu)<f(\alpha)$ fails. If $w(\beta,\kappa)\neq0$, then $\beta$ is not counted in
    $U^\snc(\kappa)$, that is, $f(\kappa)>f(\beta)$ fails. The remaining two implications
    follow from $U^\snc(\mu)=0$ and $D^\snc(\kappa)=0$.

    In case (b), the cells $\mu$ and $\kappa$ are weakly critical, so
    $U^\snc(\mu)=D^\snc(\mu)=U^\snc(\kappa)=D^\snc(\kappa)=0$. Each of the four implications
    follows from one of these equalities.
\end{proof}

\begin{lemma}\label{lem:cancelwc}
    Under the assumptions of Lemma~\ref{lem:cancelvalue}, denote by $\K'$ the Lefschetz
    complex $\K\{\mu,\kappa\}$ and by $w'$ its incidence function. Then
    $w'(\tau,\sigma)=w(\tau,\sigma)$ for any $\tau,\sigma\in\K'$ with $f(\sigma)>f(\tau)$.
    Furthermore
    \[
    U^\snc_{\K',f}(\sigma)=U^\snc_{\K,f}(\sigma),\qquad
    D^\snc_{\K',f}(\sigma)=D^\snc_{\K,f}(\sigma)
    \]
    for every $\sigma\in\K'$. In particular $f$ is a discrete Morse-Bott function on $\K'$.
    Moreover,
    \[
    \W(\K',f)=
    \begin{cases}
        \W(\K,f) & \text{in case (a)},\\
        \W(\K,f)\{\mu,\kappa\} & \text{in case (b)}.
    \end{cases}
    \]
\end{lemma}

\begin{proof}
    For $\tau,\sigma\in\K'$, put
    \[
    w_1(\tau,\sigma):=\frac{w(\tau,\kappa)w(\mu,\sigma)}{w(\mu,\kappa)}.
    \]
    By Definition~\ref{def:quotient}, $w'(\tau,\sigma)=w(\tau,\sigma)-w_1(\tau,\sigma)$.

    In case (a) we have $f(\mu)<f(\kappa)$, and in case (b) we have $f(\mu)=f(\kappa)$. In
    particular the two cases are mutually exclusive, and $f(\mu)\le f(\kappa)$ holds in both.

    Let $\tau,\sigma\in\K'$ with $f(\sigma)>f(\tau)$, and suppose that
    $w_1(\tau,\sigma)\neq0$. Then $w(\tau,\kappa)\neq0$ and $w(\mu,\sigma)\neq0$, so
    Lemma~\ref{lem:cancelvalue} and $f(\mu)\le f(\kappa)$ give
    \[
    f(\tau)\ge f(\kappa)\ge f(\mu)\ge f(\sigma)>f(\tau),
    \]
    which is a contradiction. Hence $w_1(\tau,\sigma)=0$ and $w'(\tau,\sigma)=w(\tau,\sigma)$,
    which proves the first assertion.

    Let $\sigma\in\K'$. By the first assertion,
    \[
    \{\tau\in\K'\mid w'(\tau,\sigma)\neq0,\ f(\sigma)>f(\tau)\}
    =\{\tau\in\K'\mid w(\tau,\sigma)\neq0,\ f(\sigma)>f(\tau)\},
    \]
    and by Lemma~\ref{lem:cancelvalue} neither $\mu$ nor $\kappa$ lies in
    $\{\tau\in\K\mid w(\tau,\sigma)\neq0,\ f(\sigma)>f(\tau)\}$, so the latter set coincides
    with the former one. Hence $U^\snc_{\K',f}(\sigma)=U^\snc_{\K,f}(\sigma)$, and the same
    argument gives $D^\snc_{\K',f}(\sigma)=D^\snc_{\K,f}(\sigma)$. Therefore $f$ is a
    discrete Morse-Bott function on $\K'$, and a cell of $\K'$ is weakly critical for $f$ on
    $\K'$ if and only if it is weakly critical for $f$ on $\K$.
    
    Denote by $(w')^\W$ the incidence function of $\W(\K',f)$ and by $w^\W$ the incidence
    function of $\W(\K,f)$. Let $\tau,\sigma$ be cells of $\W(\K',f)$. By
    Definition~\ref{def:wccpx},
    \begin{align*}
    (w')^\W(\tau,\sigma)&=
    \begin{cases}
        w(\tau,\sigma)-w_1(\tau,\sigma) & \text{if } f(\tau)=f(\sigma),\\
        0 & \text{otherwise},
    \end{cases}\\
    w^\W(\tau,\sigma)&=
    \begin{cases}
        w(\tau,\sigma) & \text{if } f(\tau)=f(\sigma),\\
        0 & \text{otherwise}.
    \end{cases}
    \end{align*}

    In case (a), the cells $\mu$ and $\kappa$ are not weakly critical, so $\W(\K,f)$ and
    $\W(\K',f)$ have the same cells, and their dimension functions agree. If
    $w_1(\tau,\sigma)\neq0$, then Lemma~\ref{lem:cancelvalue} and $f(\mu)<f(\kappa)$ give
    $f(\tau)\ge f(\kappa)>f(\mu)\ge f(\sigma)$, so $f(\tau)\neq f(\sigma)$. Comparing the two
    displays above, $(w')^\W$ and $w^\W$ agree. Hence $\W(\K',f)=\W(\K,f)$.

    In case (b), put $c:=f(\mu)=f(\kappa)$. The cells $\mu$ and $\kappa$ are weakly critical,
    hence are cells of $\W(\K,f)$, and $w^\W(\mu,\kappa)=w(\mu,\kappa)\neq0$ by
    Definition~\ref{def:wccpx}. Thus the Lefschetz complex $\W(\K,f)\{\mu,\kappa\}$ is
    defined, and its cells are the cells of $\W(\K,f)$ other than $\mu$ and $\kappa$, that is,
    the cells of $\W(\K',f)$. Their dimension functions agree as well. Denote by $(w^\W)'$ the
    incidence function of $\W(\K,f)\{\mu,\kappa\}$ and put
    \[
    w_1^\W(\tau,\sigma):=\frac{w^\W(\tau,\kappa)w^\W(\mu,\sigma)}{w^\W(\mu,\kappa)}.
    \]
    By Definition~\ref{def:quotient} applied to $\W(\K,f)$ and by
    Definition~\ref{def:wccpx},
    \[
    (w^\W)'(\tau,\sigma)=w^\W(\tau,\sigma)-w_1^\W(\tau,\sigma)=
    \begin{cases}
        w(\tau,\sigma)-w_1(\tau,\sigma) & \text{if } f(\tau)=f(\sigma)=c,\\
        w(\tau,\sigma) & \text{if } f(\tau)=f(\sigma)\neq c,\\
        0 & \text{if } f(\tau)\neq f(\sigma).
    \end{cases}
    \]
    Comparing this with the display for $(w')^\W$, it suffices to show that
    $f(\tau)=f(\sigma)$ and $w_1(\tau,\sigma)\neq0$ imply $f(\tau)=f(\sigma)=c$. If
    $w_1(\tau,\sigma)\neq0$, then Lemma~\ref{lem:cancelvalue} gives
    $f(\tau)\ge f(\kappa)=f(\mu)\ge f(\sigma)$, and $f(\tau)=f(\sigma)$ forces
    $f(\tau)=f(\sigma)=c$. Hence $\W(\K',f)=\W(\K,f)\{\mu,\kappa\}$.
\end{proof}

\begin{proposition}[Move~(I)]\label{prp:move1}
    Let $\mu,\kappa\in\K$ with $w(\mu,\kappa)\neq0$. Assume that either
    \begin{enumerate}[(a)]
        \item $\kappa\to_f\mu$, or
        \item $\mu$ and $\kappa$ are weakly critical and $f(\mu)=f(\kappa)$.
    \end{enumerate}
    Then $f$ is a discrete Morse-Bott function on $\K\{\mu,\kappa\}$,
    \[
    P_t\bigl(\K\{\mu,\kappa\}\bigr)=P_t(\K),\qquad
    P_t\bigl(\W(\K\{\mu,\kappa\},f)\bigr)=P_t\bigl(\W(\K,f)\bigr),
    \]
    and consequently $(\K,f)\sim(\K\{\mu,\kappa\},f)$.
\end{proposition}

\begin{proof}
    By Lemma~\ref{lem:cancelwc}, $f$ is a discrete Morse-Bott function on
    $\K\{\mu,\kappa\}$. Moreover,
    \[
    \W(\K\{\mu,\kappa\},f)=
    \begin{cases}
        \W(\K,f) & \text{in case (a)},\\
        \W(\K,f)\{\mu,\kappa\} & \text{in case (b)}.
    \end{cases}
    \]
    By Lemmas~\ref{lem:ChHoMuKappa} and~\ref{lem:rankhom},
    $P_t(\K\{\mu,\kappa\})=P_t(\K)$. In case (a), the second equality follows
    immediately from the displayed formula. In case (b), the same two lemmas applied
    to $\W(\K,f)$ give
    \[
    P_t\bigl(\W(\K,f)\{\mu,\kappa\}\bigr)=P_t\bigl(\W(\K,f)\bigr),
    \]
    and the second equality again follows from the displayed formula. The last
    assertion follows from Remark~\ref{rem:Ptcriterion}.
\end{proof}

\begin{remark}\label{rem:move1coeff}
    In the situation of Proposition~\ref{prp:move1}, the relation
    $(\K,f)\sim(\K\{\mu,\kappa\},f)$ can also be verified directly from the definition of
    $R_t$, rather than through the criterion of Remark~\ref{rem:Ptcriterion}. Let
    $k\in\Z_{\ge0}$. By the definition of $R_t$,
    \[
    R_t(\K,f)_k=\rank B_k(\K)-d^\snc_{k+1}(\K,f)-\rank B_k\bigl(\W(\K,f)\bigr),
    \]
    and the same formula holds for $(\K\{\mu,\kappa\},f)$. We compare the three terms in the
    passage from $(\K,f)$ to $(\K\{\mu,\kappa\},f)$. By (L1) we have $\dim\mu=\dim\kappa+1$,
    and Lemma~\ref{lem:rankcancel} shows that the first term drops by $1$ if $k=\dim\kappa$
    and is unchanged otherwise. By Lemma~\ref{lem:cancelwc} the number $D^\snc(\sigma)$ is
    unchanged for every cell $\sigma$ of $\K\{\mu,\kappa\}$, so the second term changes only
    through the two removed cells.

    In case (a) we have $f(\mu)<f(\kappa)$, so $\kappa$ is counted in $D^\snc(\mu)$ and $\mu$
    is counted in $U^\snc(\kappa)$. Hence $D^\snc(\mu)=1$, and $D^\snc(\kappa)=0$ by
    Lemma~\ref{lem:MBone}. The second term therefore drops by $1$ if $k+1=\dim\mu$, that is,
    if $k=\dim\kappa$, and is unchanged otherwise. The third term is unchanged, since
    $\W(\K\{\mu,\kappa\},f)=\W(\K,f)$ by Lemma~\ref{lem:cancelwc}. Hence the changes in the
    first and second terms cancel.

    In case (b) the cells $\mu$ and $\kappa$ are weakly critical, so
    $D^\snc(\mu)=D^\snc(\kappa)=0$ and the second term is unchanged. By
    Lemma~\ref{lem:cancelwc} we have $\W(\K\{\mu,\kappa\},f)=\W(\K,f)\{\mu,\kappa\}$, so
    Lemma~\ref{lem:rankcancel} applied to $\W(\K,f)$ shows that the third term drops by $1$
    if $k=\dim\kappa$ and is unchanged otherwise. Hence the changes in the first and third
    terms cancel.

    In both cases $R_t(\K,f)_k=R_t(\K\{\mu,\kappa\},f)_k$, and hence
    $R_t(\K,f)=R_t(\K\{\mu,\kappa\},f)$.
\end{remark}

\begin{example}
Figure~\ref{fig:move1} illustrates Move~(I), case (a). The left diagram shows the same
Lefschetz complex $\K$ as in Figure~\ref{fig:quotient}, whose cells are $\nu_i$,
$\sigma_i$, $\tau_i$ and $\alpha_i$ for $i\in\{1,2\}$. The second diagram displays a
discrete Morse-Bott function $f$ on $\K$ together with its strict gradient vector field
$-\nabla_s f$, which consists of the two pairs $(\tau_1,\sigma_1)$ and
$(\alpha_2,\tau_2)$. Since $\sigma_1\to_f\tau_1$, Proposition~\ref{prp:move1}, case (a),
applies with $\mu=\tau_1$ and $\kappa=\sigma_1$. The third diagram shows the triple
$\K\{\tau_1,\sigma_1\}$ of Definition~\ref{def:quotient}, on which $f$ is again a
discrete Morse-Bott function. As in Figure~\ref{fig:quotient}, the correction of the
incidence numbers makes the incidence number of the pair $(\tau_2,\sigma_2)$ zero. By
Remark~\ref{rem:quotientinc}, the value of the incidence function at
$(\alpha_2,\tau_2)$ is
unchanged, so $\tau_2\to_f\alpha_2$ still holds in $\K\{\tau_1,\sigma_1\}$ and
Proposition~\ref{prp:move1}, case (a), applies once more with $\mu=\alpha_2$ and
$\kappa=\tau_2$. The rightmost diagram shows the resulting triple
$\K\{\tau_1,\sigma_1\}\{\alpha_2,\tau_2\}$, whose cells are $\nu_1$, $\nu_2$, $\sigma_2$
and $\alpha_1$.
\end{example}

\begin{figure}[htbp]
\centering
\begin{tikzpicture}[scale=0.9]

    \begin{scope}[shift={(-3.5,0)}]
      \node (a) at (0,0) {$\nu_1$};
      \node (b) at (0,1.5) {$\sigma_1$};
      \node (c) at (0,3) {$\tau_1$};
      \node (d) at (2,0) {$\nu_2$};
      \node (e) at (2,1.5) {$\sigma_2$};
      \node (f) at (2,3) {$\tau_2$};

      \node (m) at (0,4.5) {$\alpha_1$};
      \node (n) at (2,4.5) {$\alpha_2$};

      \draw (a) -- (b) node[near end, left] {$1$};
      \draw (b) -- (c) node[near end, left] {$1$};
      \draw (d) -- (e) node[near end, right] {$-1$};
      \draw (e) -- (f) node[near end, right] {$-1$};
      \draw (a) -- (e) node[near end, above] {$1$};
      \draw (d) -- (b) node[near end, above] {$-1$};
      \draw (e) -- (c) node[near end, above] {$-1$};
      \draw (b) -- (f) node[near end, above] {$1$};

      \draw (c) -- (m) node[near end, left] {$1$};
      \draw (f) -- (n) node[near end, right] {$-1$};
      \draw (f) -- (m) node[near end, above] {$-1$};
      \draw (c) -- (n) node[near end, above] {$1$};

      \node at (1,-0.5) {$\K$};
    \end{scope}

    \begin{scope}
      \node (a) at (0,0) {$1$};
      \node (b) at (0,1.5) {$2$};
      \node (c) at (0,3) {$1$};
      \node (d) at (2,0) {$1$};
      \node (e) at (2,1.5) {$1$};
      \node (f) at (2,3) {$3$};

      \node (m) at (0,4.5) {$3$};
      \node (n) at (2,4.5) {$2$};

      \draw (a) -- (b);
      \draw[->,red,thick] (b) -- (c);
      \draw (d) -- (e);
      \draw (e) -- (f);
      \draw (a) -- (e);
      \draw (d) -- (b);
      \draw (e) -- (c);
      \draw (b) -- (f);

      \draw (c) -- (m);
      \draw[->,red,thick] (f) -- (n);
      \draw (f) -- (m);
      \draw (c) -- (n);

      \node at (1,-0.5) {$f,~-\nabla_s f$};
    \end{scope}

    \begin{scope}[shift={(3.5,0)}]
      \node (a) at (0,0) {$1$};
      \node (d) at (2,0) {$1$};
      \node (e) at (2,1.5) {$1$};
      \node (f) at (2,3) {$3$};

      \node (m) at (0,4.5) {$3$};
      \node (n) at (2,4.5) {$2$};

      \draw (d) -- (e) node[near end, right] {$-1$};
      \draw (a) -- (e) node[near end, above] {$1$};

      \draw (f) -- (n) node[near end, right] {$-1$};
      \draw (f) -- (m) node[near end, above] {$-1$};

      \node at (1,-0.5) {$\K\{\tau_1,\sigma_1\},~f$};
    \end{scope}

    \begin{scope}[shift={(7,0)}]
      \node (a) at (0,0) {$1$};
      \node (d) at (2,0) {$1$};
      \node (e) at (2,1.5) {$1$};

      \node (m) at (0,4.5) {$3$};

      \draw (d) -- (e) node[near end, right] {$-1$};
      \draw (a) -- (e) node[near end, above] {$1$};

      \node at (1,-0.5) {$\K\{\tau_1,\sigma_1\}\{\alpha_2,\tau_2\},~f$};
    \end{scope}

\end{tikzpicture}

\caption{Four diagrams illustrating Move~(I), case (a): a Lefschetz complex $\K$, a
discrete Morse-Bott function $f$ on $\K$ with its strict gradient vector field
$-\nabla_s f$, the triple $\K\{\tau_1,\sigma_1\}$, and the triple
$\K\{\tau_1,\sigma_1\}\{\alpha_2,\tau_2\}$.}
\label{fig:move1}
\end{figure}
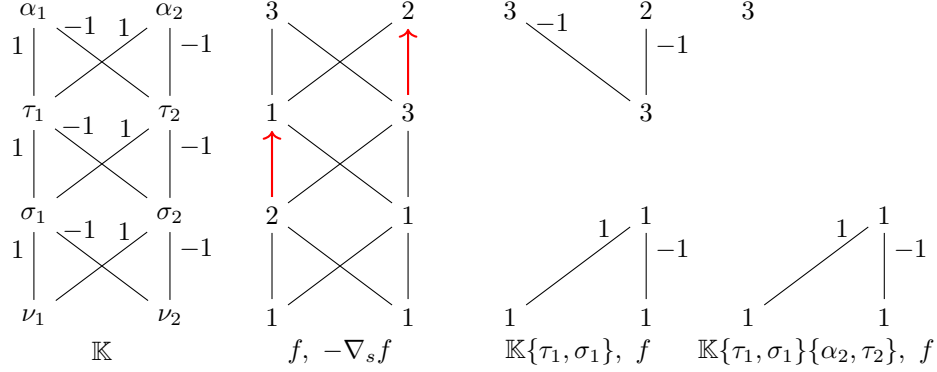

\begin{example}
Figure~\ref{fig:move1-weak} illustrates Move~(I), case (b). The left diagram shows the
same Lefschetz complex $\K$ as in Figure~\ref{fig:move1}, and the second diagram displays
a discrete Morse-Bott function $f$ on $\K$. Every cell of $\K$ is weakly critical for
$f$, so the strict gradient vector field $-\nabla_s f$ is empty. The cells $\tau_1$ and
$\sigma_1$ satisfy $w(\tau_1,\sigma_1)\neq0$ and $f(\tau_1)=f(\sigma_1)$, so
Proposition~\ref{prp:move1}, case (b), applies with $\mu=\tau_1$ and $\kappa=\sigma_1$.
The third diagram shows the triple $\K\{\tau_1,\sigma_1\}$, on which $f$ is again a
discrete Morse-Bott function. As in Figure~\ref{fig:move1}, the correction of the
incidence numbers makes the incidence number of the pair $(\tau_2,\sigma_2)$ zero. By
Lemma~\ref{lem:cancelwc} the cells $\alpha_2$ and $\tau_2$ are still weakly critical for
$f$ on $\K\{\tau_1,\sigma_1\}$, and by Remark~\ref{rem:quotientinc}, the value of the
incidence function at $(\alpha_2,\tau_2)$ is unchanged. Since
$f(\alpha_2)=f(\tau_2)$, Proposition~\ref{prp:move1}, case
(b), applies once more with $\mu=\alpha_2$ and $\kappa=\tau_2$. The rightmost diagram
shows the resulting triple $\K\{\tau_1,\sigma_1\}\{\alpha_2,\tau_2\}$, whose cells are
$\nu_1$, $\nu_2$, $\sigma_2$ and $\alpha_1$.
\end{example}

\begin{figure}[htbp]
\centering
\begin{tikzpicture}[scale=0.9]

    \begin{scope}[shift={(-3.5,0)}]
      \node (a) at (0,0) {$\nu_1$};
      \node (b) at (0,1.5) {$\sigma_1$};
      \node (c) at (0,3) {$\tau_1$};
      \node (d) at (2,0) {$\nu_2$};
      \node (e) at (2,1.5) {$\sigma_2$};
      \node (f) at (2,3) {$\tau_2$};

      \node (m) at (0,4.5) {$\alpha_1$};
      \node (n) at (2,4.5) {$\alpha_2$};

      \draw (a) -- (b) node[near end, left] {$1$};
      \draw (b) -- (c) node[near end, left] {$1$};
      \draw (d) -- (e) node[near end, right] {$-1$};
      \draw (e) -- (f) node[near end, right] {$-1$};
      \draw (a) -- (e) node[near end, above] {$1$};
      \draw (d) -- (b) node[near end, above] {$-1$};
      \draw (e) -- (c) node[near end, above] {$-1$};
      \draw (b) -- (f) node[near end, above] {$1$};

      \draw (c) -- (m) node[near end, left] {$1$};
      \draw (f) -- (n) node[near end, right] {$-1$};
      \draw (f) -- (m) node[near end, above] {$-1$};
      \draw (c) -- (n) node[near end, above] {$1$};

      \node at (1,-0.5) {$\K$};
    \end{scope}

    \begin{scope}
      \node (a) at (0,0) {$1$};
      \node (b) at (0,1.5) {$2$};
      \node (c) at (0,3) {$2$};
      \node (d) at (2,0) {$1$};
      \node (e) at (2,1.5) {$1$};
      \node (f) at (2,3) {$3$};

      \node (m) at (0,4.5) {$3$};
      \node (n) at (2,4.5) {$3$};

      \draw (a) -- (b);
      \draw (b) -- (c);
      \draw (d) -- (e);
      \draw (e) -- (f);
      \draw (a) -- (e);
      \draw (d) -- (b);
      \draw (e) -- (c);
      \draw (b) -- (f);

      \draw (c) -- (m);
      \draw (f) -- (n);
      \draw (f) -- (m);
      \draw (c) -- (n);

      \node at (1,-0.5) {$f,~-\nabla_s f=\emptyset$};
    \end{scope}

    \begin{scope}[shift={(3.5,0)}]
      \node (a) at (0,0) {$1$};
      \node (d) at (2,0) {$1$};
      \node (e) at (2,1.5) {$1$};
      \node (f) at (2,3) {$3$};

      \node (m) at (0,4.5) {$3$};
      \node (n) at (2,4.5) {$3$};

      \draw (d) -- (e) node[near end, right] {$-1$};
      \draw (a) -- (e) node[near end, above] {$1$};

      \draw (f) -- (n) node[near end, right] {$-1$};
      \draw (f) -- (m) node[near end, above] {$-1$};

      \node at (1,-0.5) {$\K\{\tau_1,\sigma_1\},~f$};
    \end{scope}

    \begin{scope}[shift={(7,0)}]
      \node (a) at (0,0) {$1$};
      \node (d) at (2,0) {$1$};
      \node (e) at (2,1.5) {$1$};

      \node (m) at (0,4.5) {$3$};

      \draw (d) -- (e) node[near end, right] {$-1$};
      \draw (a) -- (e) node[near end, above] {$1$};

      \node at (1,-0.5) {$\K\{\tau_1,\sigma_1\}\{\alpha_2,\tau_2\},~f$};
    \end{scope}

\end{tikzpicture}

\caption{Four diagrams illustrating Move~(I), case (b): a Lefschetz complex $\K$, a
discrete Morse-Bott function $f$ on $\K$ for which every cell is weakly critical, the
triple $\K\{\tau_1,\sigma_1\}$, and the triple
$\K\{\tau_1,\sigma_1\}\{\alpha_2,\tau_2\}$.}
\label{fig:move1-weak}
\end{figure}
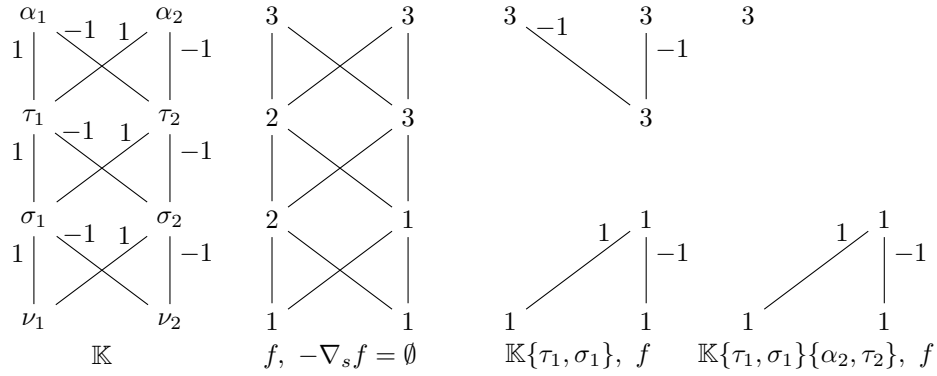

\begin{lemma}\label{lem:removeisolated}
    A cell $\eta\in\K$ is \textbf{isolated} in $\K$ if $w(\sigma,\eta)=w(\eta,\sigma)=0$ for
    all $\sigma\in\K$. For a cell $\eta$ isolated in $\K$, write $X':=X\setminus\{\eta\}$
    and define the triple $\K\{\eta\}:=(X',\dim|_{X'},w|_{X'\times X'})$. Then $\K\{\eta\}$ is
    a Lefschetz complex and
    \[
    P_t\bigl(\K\{\eta\}\bigr)=P_t(\K)-t^{\dim\eta}.
    \]
\end{lemma}

\begin{proof}
    Fix $\tau,\nu\in X'$. The only term omitted from the sum in (L2) for $\K$ is
    \[
    w(\tau,\eta)w(\eta,\nu),
    \]
    which vanishes. Hence $X'$ satisfies (L2). It follows from
    Proposition~\ref{prp:subcpx} that $\K\{\eta\}$ is a Lefschetz complex. The triple
    $(\{\eta\},\dim|_{\{\eta\}},0)$ is also a Lefschetz complex, with zero incidence
    function. Since $w$ vanishes between $\eta$ and every other cell, we have
    \[
    \K=\K\{\eta\}\sqcup(\{\eta\},\dim|_{\{\eta\}},0).
    \]
    The only nonzero chain group of $(\{\eta\},\dim|_{\{\eta\}},0)$ has rank $1$,
    and its boundary homomorphism vanishes. Thus, the Poincar\'{e} series of this
    complex is $t^{\dim\eta}$. Lemma~\ref{lem:disjointsum} now gives the asserted
    equality.
\end{proof}

\begin{proposition}[Move~(II)]\label{prp:move2}
    Let $\eta$ be a cell isolated in $\K$. Then $f$ is a discrete Morse-Bott function on
    $\K\{\eta\}$,
    \[
    P_t\bigl(\K\{\eta\}\bigr)=P_t(\K)-t^{\dim\eta},\qquad
    P_t\bigl(\W(\K\{\eta\},f)\bigr)=P_t\bigl(\W(\K,f)\bigr)-t^{\dim\eta},
    \]
    and consequently $(\K,f)\sim(\K\{\eta\},f)$.
\end{proposition}

\begin{proof}
    The cell $\eta$ is counted neither in $U^\snc(\sigma)$ nor in $D^\snc(\sigma)$ for any
    cell $\sigma$, so these two numbers are the same whether computed in $\K$ or in
    $\K\{\eta\}$. Hence $f$ is a discrete Morse-Bott function on $\K\{\eta\}$, and
    $U^\snc(\eta)=D^\snc(\eta)=0$ shows that $\eta$ is weakly critical. The complexes
    $\W(\K\{\eta\},f)$ and $\W(\K,f)\{\eta\}$ therefore have the same cells and their
    dimension functions agree. Their incidence functions agree as well, since by
    Definition~\ref{def:wccpx} both assign to a pair $\tau,\sigma$ the number $w(\tau,\sigma)$
    if $f(\tau)=f(\sigma)$ and $0$ otherwise. Hence
    $\W(\K\{\eta\},f)=\W(\K,f)\{\eta\}$. Since the incidence function of $\W(\K,f)$ takes only
    the values of $w$ and $0$, the cell $\eta$ is isolated in $\W(\K,f)$, so
    Lemma~\ref{lem:removeisolated} applied to $\K$ and to $\W(\K,f)$ gives the two displayed
    equalities, and the last assertion follows from Remark~\ref{rem:Ptcriterion}.
\end{proof}

\begin{example}
Figure~\ref{fig:move2} illustrates Move~(II). The left diagram shows a Lefschetz complex
$\K$ whose cells are $\nu_1$, $\nu_2$, $\sigma_2$, $\tau_2$, $\alpha_1$ and $\alpha_2$, and
the second diagram displays a discrete Morse-Bott function $f$ on $\K$. Every cell of $\K$
is weakly critical for $f$, so the strict gradient vector field $-\nabla_s f$ is empty. The
cell $\alpha_1$ is isolated in $\K$, so Lemma~\ref{lem:removeisolated} applies with
$\eta=\alpha_1$ and gives the triple $\K\{\alpha_1\}$ shown on the right, on which $f$ is
again a discrete Morse-Bott function by Proposition~\ref{prp:move2}. Its incidence numbers
are those of $\K$, and the Poincar\'{e} series of $\K$ and of $\W(\K,f)$ both drop by
$t^{\dim\alpha_1}$.
\end{example}

\begin{figure}[htbp]
\centering
\begin{tikzpicture}[scale=1]

    \begin{scope}[shift={(-3.5,0)}]
      \node (a) at (0,0) {$\nu_1$};
      \node (d) at (2,0) {$\nu_2$};
      \node (e) at (1,1.5) {$\sigma_2$};
      \node (f) at (1,3) {$\tau_2$};

      \node (m) at (0,4.5) {$\alpha_1$};
      \node (n) at (2,4.5) {$\alpha_2$};

      \draw (d) -- (e) node[near end, right] {$-1$};
      \draw (a) -- (e) node[near end, left] {$1$};

      \draw (f) -- (n) node[near end, right] {$-1$};

      \node at (1,-0.5) {$\K$};
    \end{scope}

    \begin{scope}
      \node (a) at (0,0) {$1$};
      \node (d) at (2,0) {$1$};
      \node (e) at (1,1.5) {$1$};
      \node (f) at (1,3) {$3$};

      \node (m) at (0,4.5) {$3$};
      \node (n) at (2,4.5) {$3$};

      \draw (d) -- (e);
      \draw (a) -- (e);

      \draw (f) -- (n);

      \node at (1,-0.5) {$f,~-\nabla_s f=\emptyset$};
    \end{scope}

    \begin{scope}[shift={(3.5,0)}]
      \node (a) at (0,0) {$1$};
      \node (d) at (2,0) {$1$};
      \node (e) at (1,1.5) {$1$};
      \node (f) at (1,3) {$3$};

      \node (n) at (2,4.5) {$3$};

      \draw (d) -- (e) node[near end, right] {$-1$};
      \draw (a) -- (e) node[near end, left] {$1$};

      \draw (f) -- (n) node[near end, right] {$-1$};

      \node at (1,-0.5) {$\K\{\alpha_1\},~f$};
    \end{scope}

\end{tikzpicture}

\caption{Three diagrams illustrating Move~(II): a Lefschetz complex $\K$ containing an
isolated cell $\alpha_1$, a discrete Morse-Bott function $f$ on $\K$ for which every cell is
weakly critical, and the triple $\K\{\alpha_1\}$.}
\label{fig:move2}
\end{figure}
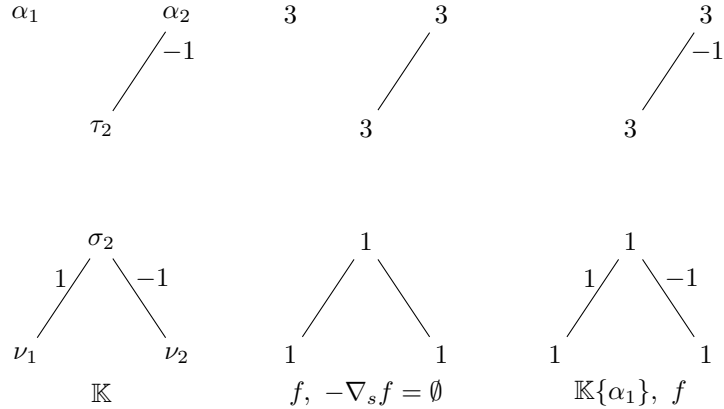

\begin{proposition}[Move~(III)]\label{prp:move3}
    Let $g\colon\K\to\R$ be a function. Assume that for any cells $\alpha,\beta\in\K$ with
    $w(\alpha,\beta)\neq0$ or $w(\beta,\alpha)\neq0$, we have
    \[
    f(\alpha)<f(\beta) \iff g(\alpha)<g(\beta).
    \]
    Then $g$ is a discrete Morse-Bott function on $\K$, $\W(\K,f)=\W(\K,g)$, and
    $(\K,f)\sim(\K,g)$.
\end{proposition}

\begin{proof}
    For any cells $\alpha,\beta\in\K$ with $w(\alpha,\beta)\neq0$, we have
    \[
    f(\alpha)=f(\beta) \iff g(\alpha)=g(\beta).
    \]
    Indeed, the ordered pair $(\beta,\alpha)$ also satisfies the hypothesis of the
    assumption, so $f(\beta)<f(\alpha)\iff g(\beta)<g(\alpha)$ holds in addition to
    $f(\alpha)<f(\beta)\iff g(\alpha)<g(\beta)$. Negating these two equivalences and
    combining them gives the displayed one.

    By the assumption, for any cell $\sigma\in\K$ we have
    \[
    \{\tau\in\K\mid w(\tau,\sigma)\neq0,\ f(\sigma)>f(\tau)\}
    =\{\tau\in\K\mid w(\tau,\sigma)\neq0,\ g(\sigma)>g(\tau)\},
    \]
    so $U^\snc_{\K,f}(\sigma)=U^\snc_{\K,g}(\sigma)$. The same argument applied to the sets
    defining $D^\snc$ gives $D^\snc_{\K,f}(\sigma)=D^\snc_{\K,g}(\sigma)$. Hence $g$ is a
    discrete Morse-Bott function on $\K$, and $f$ and $g$ have the same weakly critical
    cells.

    The complexes $\W(\K,f)$ and $\W(\K,g)$ therefore have the same cells, and their
    dimension functions agree, both being the restriction of $\dim$ to those cells. Their
    incidence functions agree as well by Definition~\ref{def:wccpx} and the equivalence
    displayed at the beginning of the proof, so $\W(\K,f)=\W(\K,g)$. The two pairs have the
    same Lefschetz complex as well, so Remark~\ref{rem:Ptcriterion} gives
    $(\K,f)\sim(\K,g)$.
\end{proof}

\begin{example}
Figure~\ref{fig:move3} illustrates Move~(III). The left diagram shows the same Lefschetz
complex $\K$ as in Figure~\ref{fig:move1}, whose cells are $\nu_i$, $\sigma_i$, $\tau_i$
and $\alpha_i$ for $i\in\{1,2\}$. The middle diagram displays a discrete Morse-Bott
function $f$ on $\K$ together with its strict gradient vector field $-\nabla_s f$, which
consists of the single pair $(\alpha_2,\tau_1)$. The function $g$ on the right orders every
pair of cells with nonzero incidence number in the same way as $f$ does, so
Proposition~\ref{prp:move3} applies and gives $\W(\K,g)=\W(\K,f)$ and $(\K,f)\sim(\K,g)$.
The two functions have the same strict gradient vector field. Moreover
$g(\sigma)=\dim\sigma$ for every cell $\sigma$ with $\dim\sigma\in\{0,1\}$, so the
hypothesis of Proposition~\ref{prp:move4} below is satisfied with $\mu=\sigma_1$ and
$\kappa=\nu_1$. The function $f$ does not satisfy that hypothesis,
since $f(\nu_1)\neq\dim\nu_1$.
\end{example}

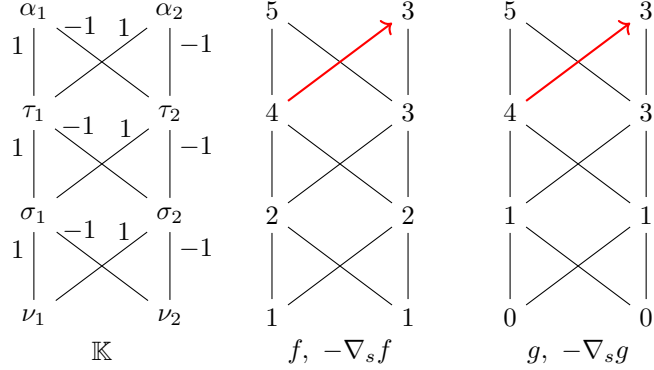
\begin{figure}[htbp]
\centering
\begin{tikzpicture}[scale=0.9]

    \begin{scope}[shift={(-3.5,0)}]
      \node (a) at (0,0) {$\nu_1$};
      \node (b) at (0,1.5) {$\sigma_1$};
      \node (c) at (0,3) {$\tau_1$};
      \node (d) at (2,0) {$\nu_2$};
      \node (e) at (2,1.5) {$\sigma_2$};
      \node (f) at (2,3) {$\tau_2$};

      \node (m) at (0,4.5) {$\alpha_1$};
      \node (n) at (2,4.5) {$\alpha_2$};

      \draw (a) -- (b) node[near end, left] {$1$};
      \draw (b) -- (c) node[near end, left] {$1$};
      \draw (d) -- (e) node[near end, right] {$-1$};
      \draw (e) -- (f) node[near end, right] {$-1$};
      \draw (a) -- (e) node[near end, above] {$1$};
      \draw (d) -- (b) node[near end, above] {$-1$};
      \draw (e) -- (c) node[near end, above] {$-1$};
      \draw (b) -- (f) node[near end, above] {$1$};

      \draw (c) -- (m) node[near end, left] {$1$};
      \draw (f) -- (n) node[near end, right] {$-1$};
      \draw (f) -- (m) node[near end, above] {$-1$};
      \draw (c) -- (n) node[near end, above] {$1$};

      \node at (1,-0.5) {$\K$};
    \end{scope}

    \begin{scope}
      \node (a) at (0,0) {$1$};
      \node (b) at (0,1.5) {$2$};
      \node (c) at (0,3) {$4$};
      \node (d) at (2,0) {$1$};
      \node (e) at (2,1.5) {$2$};
      \node (f) at (2,3) {$3$};

      \node (m) at (0,4.5) {$5$};
      \node (n) at (2,4.5) {$3$};

      \draw (a) -- (b);
      \draw (b) -- (c);
      \draw (d) -- (e);
      \draw (e) -- (f);
      \draw (a) -- (e);
      \draw (d) -- (b);
      \draw (e) -- (c);
      \draw (b) -- (f);

      \draw (c) -- (m);
      \draw (f) -- (n);
      \draw (f) -- (m);
      \draw[->,red,thick] (c) -- (n);

      \node at (1,-0.5) {$f,~-\nabla_s f$};
    \end{scope}

    \begin{scope}[shift={(3.5,0)}]
      \node (a) at (0,0) {$0$};
      \node (b) at (0,1.5) {$1$};
      \node (c) at (0,3) {$4$};
      \node (d) at (2,0) {$0$};
      \node (e) at (2,1.5) {$1$};
      \node (f) at (2,3) {$3$};

      \node (m) at (0,4.5) {$5$};
      \node (n) at (2,4.5) {$3$};

      \draw (a) -- (b);
      \draw (b) -- (c);
      \draw (d) -- (e);
      \draw (e) -- (f);
      \draw (a) -- (e);
      \draw (d) -- (b);
      \draw (e) -- (c);
      \draw (b) -- (f);

      \draw (c) -- (m);
      \draw (f) -- (n);
      \draw (f) -- (m);
      \draw[->,red,thick] (c) -- (n);

      \node at (1,-0.5) {$g,~-\nabla_s g$};
    \end{scope}

\end{tikzpicture}

\caption{Three diagrams illustrating Move~(III): a Lefschetz complex $\K$, a discrete
Morse-Bott function $f$ on $\K$ with its strict gradient vector field $-\nabla_s f$, and
the function $g$ obtained from $f$ by Move~(III), which agrees with $\dim$ on the cells of
dimension $0$ and $1$.}
\label{fig:move3}
\end{figure}

\begin{definition}
  Let $\mu,\kappa\in\K$ with $\dim\mu-\dim\kappa=1$. We define the triple $E(\mu,\kappa):=(\{\mu,\kappa\},\dim|_{\{\mu,\kappa\}},w_E)$, where $w_E(\mu,\kappa)=1$ and $w_E(\tau,\sigma)=0$ otherwise. We call $E(\mu,\kappa)$ an \textbf{elementary block}.
\end{definition}

\begin{lemma}\label{lem:Rtblock}
    Let $\mu,\kappa\in\K$ with $\dim\mu-\dim\kappa=1$. Then $E(\mu,\kappa)$ is a Lefschetz
    complex, $P_t(E(\mu,\kappa))=0$ and $R_t(E(\mu,\kappa),\dim)=t^{\dim\mu-1}$.
\end{lemma}

\begin{proof}
    Put $p:=\dim\mu$. The only pair on which $w_E$ is nonzero is $(\mu,\kappa)$, and
    $\dim\mu-\dim\kappa=1$, so condition (L1) holds. There are no cells
    $\tau,\sigma,\nu\in\{\mu,\kappa\}$ with $w_E(\tau,\sigma)\neq0$ and
    $w_E(\sigma,\nu)\neq0$, since such a cell $\sigma$ would equal both $\kappa$ and $\mu$.
    Hence $\sum_{\sigma\in\{\mu,\kappa\}}w_E(\tau,\sigma)w_E(\sigma,\nu)=0$ for any
    $\tau,\nu\in\{\mu,\kappa\}$, so condition (L2) holds. Since $\{\mu,\kappa\}$ is finite,
    $E(\mu,\kappa)$ is a Lefschetz complex.

    The boundary homomorphism of $E(\mu,\kappa)$ satisfies $\partial_p(\mu)=\kappa$ and
    $\partial_k=0$ for $k\neq p$. Hence $H_k(E(\mu,\kappa))=\{0\}$ for every $k$,
    $\rank B_{p-1}(E(\mu,\kappa))=1$ and $\rank B_k(E(\mu,\kappa))=0$ for $k\neq p-1$. The
    assertion follows from Lemma~\ref{lem:dimMB} and the definition of $R_t$.
\end{proof}

\begin{lemma}\label{lem:move4wc}
    Let $p\in\Z_{\ge1}$. Assume that $f$ agrees with $\dim$ on the cells of dimension
    $p-1$ and $p$. Then the following hold.
    \begin{enumerate}[(1)]
        \item If a $p$-cell $\sigma$ is not weakly critical, then $\partial_p^\K(\sigma)=0$.
        \item If a $(p-1)$-cell $\nu$ is not weakly critical, then $w(\sigma,\nu)=0$ for
              every $p$-cell $\sigma$.
    \end{enumerate}
\end{lemma}

\begin{proof}
    (1) Let $\sigma$ be a $p$-cell that is not weakly critical. Since every $(p-1)$-cell
    has $f$-value $p-1$ and $f(\sigma)=p$, we have $D^\snc(\sigma)=0$, and therefore
    $U^\snc(\sigma)\ge1$. Thus there is a cell $\tau\in\K$ such that $\sigma\to_f\tau$.
    Suppose there is a $p$-cell $\widetilde\sigma\neq\sigma$ such that
    $w(\tau,\widetilde\sigma)\neq0$. Then $f(\tau)<f(\sigma)=f(\widetilde\sigma)$, so both
    $\sigma$ and $\widetilde\sigma$ are counted in $D^\snc(\tau)$, which contradicts
    $D^\snc(\tau)\le1$. Therefore $\partial_{p+1}^\K(\tau)=w(\tau,\sigma)\sigma$, and so
    $\partial_p^\K(\sigma)=0$ by $\partial_p^\K\circ\partial_{p+1}^\K=0$.

    (2) Let $\nu$ be a $(p-1)$-cell that is not weakly critical. Since every $p$-cell has
    $f$-value $p$ and $f(\nu)=p-1$, we have $U^\snc(\nu)=0$, and therefore
    $D^\snc(\nu)\ge1$. Thus there is a cell $\alpha\in\K$ such that $\alpha\to_f\nu$.
    Suppose there is a $p$-cell $\sigma$ such that $w(\sigma,\nu)\neq0$. Since
    $w(\sigma,\nu)w(\nu,\alpha)\neq0$, Lemma~\ref{lem:incidence} gives a cell
    $\widetilde\nu\neq\nu$ such that $w(\sigma,\widetilde\nu)w(\widetilde\nu,\alpha)\neq0$.
    By (L1) it is a $(p-1)$-cell, so $f(\widetilde\nu)=p-1$. Then
    $f(\alpha)>f(\nu)=f(\widetilde\nu)$, so both $\nu$ and $\widetilde\nu$ are counted in
    $U^\snc(\alpha)$, which contradicts $U^\snc(\alpha)\le1$. Hence $w(\sigma,\nu)=0$ for
    every $p$-cell $\sigma$.
\end{proof}

\begin{lemma}\label{lem:move4quotient}
    Let $\mu,\kappa\in\K$ with $w(\mu,\kappa)\neq0$, and assume that $f$ agrees with $\dim$
    on the cells of dimension $\dim\mu$ and $\dim\kappa$. Then $\mu$ and $\kappa$ are weakly
    critical, the function $f$ is a discrete Morse-Bott function on $\K\{\mu,\kappa\}$, and
    $\W(\K\{\mu,\kappa\},f)$ is obtained from $\W(\K,f)$ by removing $\mu$ and $\kappa$ and
    restricting the incidence function.
\end{lemma}

\begin{proof}
    Put $p:=\dim\mu$, so that $\dim\kappa=p-1$ by (L1). Since $w(\mu,\kappa)\neq0$, we have
    $\partial_p^\K(\mu)\neq0$. Hence $\mu$ is weakly critical by
    Lemma~\ref{lem:move4wc}~(1), and $\kappa$ is weakly critical by
    Lemma~\ref{lem:move4wc}~(2).

    Since $\mu$ and $\kappa$ are weakly critical, neither of them is counted in
    $U^\snc_{\K,f}(\sigma)$ or in $D^\snc_{\K,f}(\sigma)$ for any cell $\sigma\in\K$.

    Write $\K':=\K\{\mu,\kappa\}$ and denote by $w'$ its incidence function. For
    $\tau,\sigma\in\K'$, the function $w'$ differs from $w$ at $(\tau,\sigma)$ only if
    $w(\tau,\kappa)w(\mu,\sigma)\neq0$, that is, only if $\dim\tau=p$ and $\dim\sigma=p-1$.
    In that case $f(\tau)=p>p-1=f(\sigma)$, and therefore
    \[
    f(\tau)\le f(\sigma)\implies w'(\tau,\sigma)=w(\tau,\sigma)
    \]
    for $\tau,\sigma\in\K'$. By the implication above, the incidence numbers occurring in
    $U^\snc$ and in $D^\snc$ are the same in $\K'$ as in $\K$, and by the previous paragraph
    neither $\mu$ nor $\kappa$ contributes to them. Hence $U^\snc_{\K',f}=U^\snc_{\K,f}$ and
    $D^\snc_{\K',f}=D^\snc_{\K,f}$ on $\K'$, so $f$ is a discrete Morse-Bott function on
    $\K'$ and a cell of $\K'$ is weakly critical for $f$ on $\K'$ if and only if it is
    weakly critical for $f$ on $\K$. The complexes $\W(\K',f)$ and $\W(\K,f)$ therefore have
    the same cells apart from $\mu$ and $\kappa$, and their dimension functions agree. Since
    $f(\tau)=f(\sigma)$ implies $f(\tau)\le f(\sigma)$, the implication above and
    Definition~\ref{def:wccpx} show that their incidence functions agree as well.
\end{proof}

Under the hypothesis of the next proposition, the two cells to be cancelled are weakly
critical and take different $f$-values, so they lie in different weakly critical sets.
Cancelling them therefore lowers $P_t(\W(\K,f))$ by the two terms corresponding to their
dimensions, while $P_t(\K)$ is unchanged. In case~(b) of Move~(I) the two cells lie in a
single weakly critical set, and the cancellation is instead matched by a cancellation
inside $\W(\K,f)$. The elementary block determined by the cancelled pair is adjoined to
restore the balance. Its Poincar\'{e} series vanishes, so adjoining it leaves $P_t(\K)$
unchanged. The incidence function of its weakly critical complex vanishes as well, and
that complex contributes exactly the two terms that were lost.

\begin{proposition}[Move~(IV)]\label{prp:move4}
    Let $\mu,\kappa\in\K$ with $w(\mu,\kappa)\neq0$, and assume that $f$ agrees with $\dim$
    on the cells of dimension $\dim\mu$ and $\dim\kappa$. Put
    $\K_E:=\K\{\mu,\kappa\}\sqcup E(\mu,\kappa)$. Then $f$ is a discrete Morse-Bott function
    on $\K_E$,
    \[
    P_t(\K_E)=P_t(\K),\qquad P_t\bigl(\W(\K_E,f)\bigr)=P_t\bigl(\W(\K,f)\bigr),
    \]
    and consequently $(\K,f)\sim(\K_E,f)$.
\end{proposition}

\begin{proof}
    Put $p:=\dim\mu$, so that $\dim\kappa=p-1$ by (L1), and write
    $\K':=\K\{\mu,\kappa\}$, $\W:=\W(\K,f)$ and $\W':=\W(\K',f)$.

    By Lemma~\ref{lem:move4quotient}, the function $f$ is a discrete Morse-Bott function on
    $\K'$, and $\W'$ is obtained from $\W$ by removing $\mu$ and $\kappa$ and restricting
    the incidence function. Since $f=\dim$ on $E(\mu,\kappa)$, Lemmas~\ref{lem:dimMB}
    and~\ref{lem:Rtdisjoint} show that $f$ is a discrete Morse-Bott function on $\K_E$. By
    Lemmas~\ref{lem:ChHoMuKappa},~\ref{lem:rankhom},~\ref{lem:Rtblock},
    and~\ref{lem:disjointsum},
    \[
    P_t(\K_E)=P_t(\K')+P_t(E(\mu,\kappa))=P_t(\K).
    \]

    It remains to prove
    \begin{equation}\label{eq:move4W}
    P_t(\W')=P_t(\W)-t^p-t^{p-1}.
    \end{equation}
    In the reduction that follows, the symbols $\K'$, $\W$ and $\W'$ refer to the current
    complex.

    Since $f$ takes the value $p$ on the $p$-cells and the value $p-1$ on the
    $(p-1)$-cells, we have $w^\W(\mu,\beta)=0$ and $w^\W(\beta,\kappa)=0$ for every cell
    $\beta$ of $\W$. Suppose that there is a cell $\tau$ of $\W$ with $w^\W(\tau,\mu)\neq0$
    or a cell $\alpha$ of $\W$ with $w^\W(\kappa,\alpha)\neq0$. We claim that there is then
    a pair of cells of $\W$ to which Lemma~\ref{lem:cancelvalue} applies in case (b) and
    whose second entry $\gamma$ satisfies $w(\mu,\gamma)=0$.

    In the first case $w(\tau,\mu)w(\mu,\kappa)\neq0$, so Lemma~\ref{lem:incidence} gives a
    cell $\widetilde\mu\neq\mu$ with $w(\tau,\widetilde\mu)w(\widetilde\mu,\kappa)\neq0$. By
    (L1) it is a $p$-cell, so $\widetilde\mu\neq\kappa$, $w(\mu,\widetilde\mu)=0$ and
    $f(\widetilde\mu)=p=f(\tau)$. Applying Lemma~\ref{lem:move4quotient} to the pair
    $(\widetilde\mu,\kappa)$ shows that $\widetilde\mu$ is weakly critical, so
    $(\tau,\widetilde\mu)$ proves the claim.

    In the second case $w(\mu,\kappa)w(\kappa,\alpha)\neq0$, so Lemma~\ref{lem:incidence}
    gives a cell $\widetilde\kappa\neq\kappa$ with
    $w(\mu,\widetilde\kappa)w(\widetilde\kappa,\alpha)\neq0$. By (L1) it is a $(p-1)$-cell,
    so $\widetilde\kappa\neq\mu$, $w(\mu,\alpha)=0$ and
    $f(\widetilde\kappa)=p-1=f(\alpha)$. Applying Lemma~\ref{lem:move4quotient} to the pair
    $(\mu,\widetilde\kappa)$ shows that $\widetilde\kappa$ is weakly critical, so
    $(\widetilde\kappa,\alpha)$ proves the claim.

    Let $(\beta,\gamma)$ be a pair as in the claim and put $\K_1:=\K\{\beta,\gamma\}$. Since
    $w(\mu,\gamma)=0$, Remark~\ref{rem:quotientinc} shows that the value of the incidence
    function at $(\mu,\kappa)$ is unchanged in $\K_1$, so the hypothesis of the proposition holds for
    $\K_1$. By Lemma~\ref{lem:cancelwc} we have $\W(\K_1,f)=\W\{\beta,\gamma\}$, and
    Lemma~\ref{lem:move4quotient} applied to $\K_1$ gives
    \[
    \W(\K_1\{\mu,\kappa\},f)=\W(\K_1,f)\setminus\{\mu,\kappa\}
    =\W\{\beta,\gamma\}\setminus\{\mu,\kappa\}=\W'\{\beta,\gamma\}.
    \]
    Here $\beta$ and $\gamma$ are cells of $\W'$, and $\W'$ is obtained from $\W$ by
    removing $\mu$ and $\kappa$ and restricting the incidence function, so the correction of
    Definition~\ref{def:quotient} takes the same value at a pair of cells of $\W'$ whether
    it is computed in $\W$ or in $\W'$. Hence Lemmas~\ref{lem:ChHoMuKappa}
    and~\ref{lem:rankhom}, applied to $\W$ and to $\W'$, show that both Poincar\'{e} series
    in \eqref{eq:move4W} are unchanged by this replacement. Each replacement removes a cell
    of dimension at least $p-2$ and at most $p+1$, and there are finitely many such cells,
    so we may assume that $\mu$ and $\kappa$ are isolated in $\W$.

    In this case Lemma~\ref{lem:move4quotient} gives
    \[
    \W=\W'\sqcup(\{\mu\},\dim|_{\{\mu\}},0)\sqcup(\{\kappa\},\dim|_{\{\kappa\}},0),
    \]
    so $P_t(\W)=P_t(\W')+t^p+t^{p-1}$ by Lemma~\ref{lem:disjointsum}, which is
    \eqref{eq:move4W}.

    Since the incidence function of $\W(E(\mu,\kappa),f)$ vanishes by
    Lemma~\ref{lem:dimMB}, we have $P_t(\W(E(\mu,\kappa),f))=t^p+t^{p-1}$. Hence
    $P_t(\W(\K_E,f))=P_t(\W')+t^p+t^{p-1}=P_t(\W)$ by Lemmas~\ref{lem:Rtdisjoint}
    and~\ref{lem:disjointsum}. The last assertion follows from
    Remark~\ref{rem:Ptcriterion}.
\end{proof}

\begin{example}
Figure~\ref{fig:move4} illustrates Move~(IV). The left diagram shows the same Lefschetz
complex $\K$ as in Figure~\ref{fig:move3}, and the middle diagram displays the discrete
Morse-Bott function $g$ of Figure~\ref{fig:move3} together with its strict gradient vector
field $-\nabla_s g$, which consists of the single pair $(\alpha_2,\tau_1)$. Since
$w(\sigma_1,\nu_1)\neq0$ and $g$ agrees with $\dim$ on the cells of dimension $0$ and $1$,
Proposition~\ref{prp:move4} applies with $\mu=\sigma_1$ and $\kappa=\nu_1$. The right
diagram shows $\K_E=\K\{\sigma_1,\nu_1\}\sqcup E(\sigma_1,\nu_1)$, whose second summand is
the pair of cells drawn on the right, and on which $g$ is again a discrete Morse-Bott
function. As in Figure~\ref{fig:move1}, the correction of the incidence numbers makes the
incidence number of the pair $(\sigma_2,\nu_2)$ zero.
\end{example}

\begin{figure}[htbp]
\centering
\begin{tikzpicture}[scale=0.9]

    \begin{scope}[shift={(-3.5,0)}]
      \node (a) at (0,0) {$\nu_1$};
      \node (b) at (0,1.5) {$\sigma_1$};
      \node (c) at (0,3) {$\tau_1$};
      \node (d) at (2,0) {$\nu_2$};
      \node (e) at (2,1.5) {$\sigma_2$};
      \node (f) at (2,3) {$\tau_2$};

      \node (m) at (0,4.5) {$\alpha_1$};
      \node (n) at (2,4.5) {$\alpha_2$};

      \draw (a) -- (b) node[near end, left] {$1$};
      \draw (b) -- (c) node[near end, left] {$1$};
      \draw (d) -- (e) node[near end, right] {$-1$};
      \draw (e) -- (f) node[near end, right] {$-1$};
      \draw (a) -- (e) node[near end, above] {$1$};
      \draw (d) -- (b) node[near end, above] {$-1$};
      \draw (e) -- (c) node[near end, above] {$-1$};
      \draw (b) -- (f) node[near end, above] {$1$};

      \draw (c) -- (m) node[near end, left] {$1$};
      \draw (f) -- (n) node[near end, right] {$-1$};
      \draw (f) -- (m) node[near end, above] {$-1$};
      \draw (c) -- (n) node[near end, above] {$1$};

      \node at (1,-0.5) {$\K$};
    \end{scope}

    \begin{scope}
      \node (a) at (0,0) {$0$};
      \node (b) at (0,1.5) {$1$};
      \node (c) at (0,3) {$4$};
      \node (d) at (2,0) {$0$};
      \node (e) at (2,1.5) {$1$};
      \node (f) at (2,3) {$3$};

      \node (m) at (0,4.5) {$5$};
      \node (n) at (2,4.5) {$3$};

      \draw (a) -- (b);
      \draw (b) -- (c);
      \draw (d) -- (e);
      \draw (e) -- (f);
      \draw (a) -- (e);
      \draw (d) -- (b);
      \draw (e) -- (c);
      \draw (b) -- (f);

      \draw (c) -- (m);
      \draw (f) -- (n);
      \draw (f) -- (m);
      \draw[->,red,thick] (c) -- (n);

      \node at (1,-0.5) {$g,~-\nabla_s g$};
    \end{scope}

    \begin{scope}[shift={(3.5,0)}]
      \node (c) at (0,3) {$4$};
      \node (m) at (0,4.5) {$5$};
      \node (d) at (2,0) {$0$};
      \node (e) at (2,1.5) {$1$};
      \node (f) at (2,3) {$3$};
      \node (n) at (2,4.5) {$3$};

      \node (p) at (3.8,0) {$0$};
      \node (q) at (3.8,1.5) {$1$};

      \draw (e) -- (f) node[near end, right] {$-1$};
      \draw (e) -- (c) node[near end, above] {$-1$};

      \draw (c) -- (m) node[near end, left] {$1$};
      \draw (f) -- (n) node[near end, right] {$-1$};
      \draw (f) -- (m) node[near end, above] {$-1$};
      \draw (c) -- (n) node[near end, above] {$1$};

      \draw (p) -- (q) node[near end, right] {$1$};

      \node at (1,-0.5) {$\K\{\sigma_1,\nu_1\}$};
      \node at (3.8,-0.5) {$E(\sigma_1,\nu_1)$};
    \end{scope}

\end{tikzpicture}

\caption{Three diagrams illustrating Move~(IV): a Lefschetz complex $\K$, the discrete
Morse-Bott function $g$ of Figure~\ref{fig:move3} with its strict gradient vector field
$-\nabla_s g$, and the two summands of
$\K_E=\K\{\sigma_1,\nu_1\}\sqcup E(\sigma_1,\nu_1)$.}
\label{fig:move4}
\end{figure}
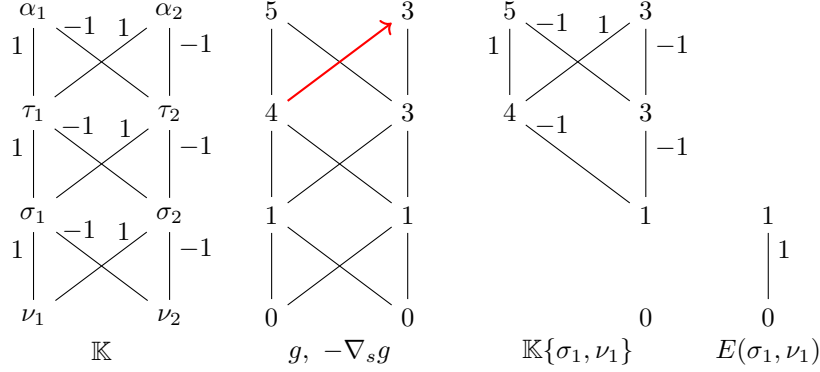

\begin{remark}\label{rem:moveonsummand}
    Let $\K=\K_1\sqcup\K_2$ be a disjoint union of Lefschetz complexes, let $f\colon\K\to\R$
    be a discrete Morse-Bott function, and suppose that the data of a Move lie in $\K_1$,
    that is, $\mu,\kappa\in\K_1$ for Moves~(I) and~(IV), $\eta\in\K_1$ for Move~(II), and
    $g|_{\K_2}=f|_{\K_2}$ for Move~(III). The Move is applied to $\K$ and its hypothesis is a
    condition on $\K$. We record when this condition may be verified inside $\K_1$.

    If $w(\tau,\sigma)\neq0$ in $\K$, then $\tau$ and $\sigma$ lie in the same summand.
    Hence $U^\snc_{\K,f}(\sigma)=U^\snc_{\K_1,f}(\sigma)$ and
    $D^\snc_{\K,f}(\sigma)=D^\snc_{\K_1,f}(\sigma)$ for every $\sigma\in\K_1$, and the
    correction term $w(\tau,\kappa)w(\mu,\sigma)/w(\mu,\kappa)$ vanishes whenever $\tau$ or
    $\sigma$ lies in $\K_2$, so that
    \[
    \K\{\mu,\kappa\}=\K_1\{\mu,\kappa\}\sqcup\K_2.
    \]

    The hypotheses of Moves~(I) and~(II) are conditions on incidence numbers and on the
    values of $f$, $U^\snc$ and $D^\snc$ at cells of $\K_1$, and the incidence numbers
    between the two summands vanish. They may therefore be verified inside $\K_1$. For
    Move~(III), the hypothesis holds at every pair of cells of $\K_2$ because $g$ agrees with
    $f$ there, so it suffices to verify it inside $\K_1$. The hypothesis of Move~(IV)
    requires $f$ to agree with $\dim$ at every cell of $\K$ of dimension $\dim\mu$ or
    $\dim\kappa$, and nothing about the cells of $\K_2$ follows from the
    corresponding condition on $\K_1$. It must therefore be verified on $\K_2$ as well.

    In each case the Move leaves $\K_2$ unchanged and acts on $\K_1$ alone. For Move~(IV)
    this means that $\K$ is replaced by $\K_1\{\mu,\kappa\}\sqcup E(\mu,\kappa)\sqcup\K_2$.
\end{remark}

For a Lefschetz complex $\L$ and $q\in\Z_{\ge0}$, write $\L^{\ge q}$ and $\L^{\le q}$ for the
sets of cells of dimension at least $q$ and at most $q$, respectively. If $\tau$ and $\nu$
both lie in one of these sets, then so does every cell $\sigma$ with
$w(\tau,\sigma)w(\sigma,\nu)\neq0$, by (L1). Hence the sum in (L2) is unchanged, and
Proposition~\ref{prp:subcpx} shows that both are subcomplexes of $\L$.

\begin{lemma}\label{lem:stage}
    Let $n\in\Z_{\ge0}$, let $\M$ be a Lefschetz complex with $\M=\M^{\ge n}$, and let $g$ be
    a discrete Morse-Bott function on $\M$. Then there are a Lefschetz complex $\M'$ with
    $\M'=(\M')^{\ge n+1}$, a discrete Morse-Bott function $g'$ on $\M'$, and a Lefschetz
    complex $\E$ which is a disjoint union of finitely many elementary blocks, each
    consisting of one $n$-cell and one $(n+1)$-cell, such that
    \[
    (\M,g)\sim(\E\sqcup\M',\ \dim\sqcup g').
    \]
    Moreover $P_t(\M)=P_t(\M')+b_n^{\M}\,t^n$, and $b_n^{\M}$ equals the number of cells
    removed by Move~(II) in the construction below.
\end{lemma}

\begin{proof}
    Throughout, $U$, $D$, $U^\snc$ and $D^\snc$ refer to the current complex and function.
    For $j=n,n+1,n+2$, while $u_j^\snc>0$, choose a $j$-cell $\sigma$ with
    $U^\snc(\sigma)=1$ and let $\tau$ be the $(j+1)$-cell counted in it. Then
    $\sigma\to_g\tau$, so Move~(I), case~(a), applies to $(\tau,\sigma)$. By
    Lemma~\ref{lem:cancelwc}, each application decreases $u_j^\snc$ by exactly one and does
    not change $U^\snc$ or $D^\snc$ on the remaining cells. Hence, after finitely many
    applications, we obtain
    \[
    u_n^\snc=u_{n+1}^\snc=u_{n+2}^\snc=0.
    \]
    Denote the resulting complex by $\M_{\mathrm{wc}}$. We have $d_n^\snc=0$ because
    $\M_{\mathrm{wc}}$ has no $(n-1)$-cells, and Lemma~\ref{lem:countMB} gives
    \[
    d_{n+1}^\snc=u_n^\snc=0,\qquad d_{n+2}^\snc=u_{n+1}^\snc=0.
    \]
    Thus every cell of $\M_{\mathrm{wc}}$ of dimension at most $n+2$ is weakly critical.

    Starting from $\M_{\mathrm{wc}}$, apply Move~(I), case~(b), whenever there is an
    equal-valued pair with nonzero incidence number whose dimensions are $(n+2,n+1)$ or
    $(n+1,n)$. This is possible because all cells involved are weakly critical. By
    Lemma~\ref{lem:cancelwc}, the values of $U^\snc$ and $D^\snc$ are unchanged on the
    remaining cells. Moreover, the removed cells are weakly critical and therefore contribute to
    none of $u_k^\snc$ and $d_k^\snc$, so these numbers are unchanged for every $k$. Each
    application removes an $n$-cell or an $(n+1)$-cell, and hence the process terminates after
    finitely many applications. Denote the resulting complex by $\M_{\mathrm{cr}}$.

    Every $n$-cell and every $(n+1)$-cell of $\M_{\mathrm{cr}}$ is critical for $g$. Indeed,
    these cells remain weakly critical. Hence, if an $n$-cell or an $(n+1)$-cell $\sigma$
    were not critical, a cell counted in $U(\sigma)$ or $D(\sigma)$ would have the same
    $g$-value as $\sigma$. Since there are no $(n-1)$-cells, the cell $\sigma$ would then
    belong to an equal-valued pair of dimensions $(n+1,n)$ or $(n+2,n+1)$, contradicting the
    termination condition.

    If $\M_{\mathrm{cr}}$ has a cell of dimension $n+2$, put
    $c:=(n+2)-\min\{g(\tau')\mid \dim\tau'=n+2\}$, which is well defined since there are
    finitely many $(n+2)$-cells. Otherwise put $c:=0$. Define
    $g'\colon\M_{\mathrm{cr}}\to\R$ by
    $g'(\sigma):=\dim\sigma$ for $\dim\sigma\le n+1$ and $g'(\sigma):=g(\sigma)+c$ for
    $\dim\sigma\ge n+2$. Let $\alpha,\beta$ be cells with $w(\alpha,\beta)\neq0$. If
    $\dim\alpha\le n+1$, then $g(\beta)<g(\alpha)$ because $\alpha$ is critical for $g$, and
    $g'(\beta)=\dim\beta<\dim\alpha=g'(\alpha)$. If $\dim\alpha=n+2$, then
    $g(\beta)<g(\alpha)$ because $\beta$ is critical for $g$, and
    $g'(\beta)=n+1<n+2\le g'(\alpha)$ by the choice of $c$. If $\dim\alpha\ge n+3$, then $g'$
    differs from $g$ by the constant $c$ at both cells. In every case
    $g(\beta)<g(\alpha)$ holds if and only if $g'(\beta)<g'(\alpha)$, and
    $g(\alpha)<g(\beta)$ holds if and only if $g'(\alpha)<g'(\beta)$. Hence Move~(III)
    applies and $(\M_{\mathrm{cr}},g)\sim(\M_{\mathrm{cr}},g')$.

    The function $g'$ takes the value $n+1$ at every $(n+1)$-cell and the value $n$ at every
    $n$-cell. Write $\widetilde{\M}$ for the summand of the current complex other than the
    elementary blocks already split off, so that $\widetilde{\M}=\M_{\mathrm{cr}}$ initially.
    While there are an $n$-cell $\kappa$ and an $(n+1)$-cell $\mu$ of
    $\widetilde{\M}$ with
    $w(\mu,\kappa)\neq0$, apply
    Move~(IV) to $(\mu,\kappa)$. Its hypothesis holds on the whole current complex. On
    $\widetilde{\M}$, the function $g'$ agrees with $\dim$ at every $n$-cell and
    $(n+1)$-cell, and on
    each elementary block already split off it agrees with $\dim$. By
    Remark~\ref{rem:moveonsummand}, this replaces $\widetilde{\M}$ by
    $\widetilde{\M}\{\mu,\kappa\}\sqcup E(\mu,\kappa)$. We then replace
    $\widetilde{\M}$ by $\widetilde{\M}\{\mu,\kappa\}$. The $n$-cells and
    $(n+1)$-cells of the new $\widetilde{\M}$ are those of the previous one other than
    $\kappa$ and $\mu$, and $g'$ agrees with $\dim$ on the newly adjoined elementary block.
    Thus the same global hypothesis is satisfied before the next application. Each
    application removes one $n$-cell, so the process terminates.

    Every $n$-cell $\eta$ of $\widetilde{\M}$ is isolated in $\widetilde{\M}$. Indeed, by
    the termination condition of the preceding procedure, $w(\mu,\eta)=0$ for every
    $(n+1)$-cell $\mu$. Hence $w(\tau,\eta)=0$ for every cell $\tau$ by (L1), and
    $w(\eta,\nu)=0$ for every cell $\nu$ because $\widetilde{\M}$ has no $(n-1)$-cells.
    By Remark~\ref{rem:moveonsummand}, we may therefore apply Move~(II) to every $n$-cell of
    $\widetilde{\M}$. The process terminates because there are finitely many $n$-cells.

    Let $\E$ be the disjoint union of the elementary blocks produced by the preceding
    applications of Move~(IV), and let $\M'$ be the complex remaining after these
    applications of Move~(II). Every elementary block in $\E$ consists of one $n$-cell and
    one $(n+1)$-cell, and every cell of $\M'$ has dimension at least $n+1$. Moreover, $g'$ is
    a discrete Morse-Bott function on $\M'$, and $g'$ coincides with $\dim$ on $\E$. By
    Propositions~\ref{prp:move1},~\ref{prp:move2},~\ref{prp:move3}
    and~\ref{prp:move4}, every Move used above preserves the
    equivalence class of the Morse-Bott pair. The transitivity of $\sim$ therefore gives
    \[
    (\M,g)\sim(\E\sqcup\M',\ \dim\sqcup g').
    \]

    Finally, we compare Poincar\'{e} series. By
    Propositions~\ref{prp:move1},~\ref{prp:move3} and~\ref{prp:move4},
    all applications of Moves~(I), (III) and~(IV) made above leave the
    Poincar\'{e} series of the whole complex unchanged. Each application
    of Move~(II) lowers it by $t^n$ by Proposition~\ref{prp:move2}, since
    the removed cell is an $n$-cell. Writing $r$ for the number of these
    applications, we obtain
    \[
    P_t(\M)=P_t(\E\sqcup\M')+r\,t^n.
    \]
    By Lemmas~\ref{lem:disjointsum} and~\ref{lem:Rtblock},
    \[
    P_t(\E\sqcup\M')=P_t(\E)+P_t(\M')=P_t(\M').
    \]
    Moreover, every cell of $\M'$ has dimension at least $n+1$, and hence
    $b_n^{\M'}=0$. Comparing the coefficients of $t^n$ therefore gives
    $r=b_n^{\M}$.
\end{proof}

\begin{proposition}\label{prp:stage}
    For every $n\in\Z_{\ge0}$ there are a Lefschetz complex $\E_n$ with finitely many
    cells, which is a disjoint union of
    elementary blocks, a Lefschetz complex $\M_n$ with $\M_n=\M_n^{\ge n}$, and a discrete
    Morse-Bott function $f_n$ on $\M_n$, such that $\E_n\subset\E_{n+1}$, every cell of
    $\E_{n+1}\setminus\E_n$ has dimension $n$ or $n+1$, and
    \[
    (\K,f)\sim(\E_n\sqcup\M_n,\ \dim\sqcup f_n).
    \]
    Moreover $P_t(\K)=P_t(\M_n)+\sum_{m=0}^{n-1}b_m^{\M_m}\,t^m$.
\end{proposition}

\begin{proof}
    Put $\E_0:=\emptyset$, $\M_0:=\K$ and $f_0:=f$. Assume that the data have been
    constructed for $n$. Apply each Move in the construction of
    Lemma~\ref{lem:stage} for $(\M_n,f_n)$ to the whole complex
    $\E_n\sqcup\M_n$, leaving the summand $\E_n$ unchanged. By
    Remark~\ref{rem:moveonsummand}, the hypotheses of Moves~(I), (II) and~(III)
    may be verified in the summand obtained from $\M_n$. Only the hypothesis of
    Move~(IV) must also be verified on $\E_n$. The function on $\E_n$ is $\dim$,
    so this hypothesis holds there. Together with the verification made at each
    application of Move~(IV) in the proof of Lemma~\ref{lem:stage}, this verifies
    the hypothesis on the whole current complex.

    Apply Lemma~\ref{lem:stage} to $(\M_n,f_n)$ and put
    $\E_{n+1}:=\E_n\sqcup\E$, $\M_{n+1}:=\M'$ and $f_{n+1}:=g'$. By
    Lemma~\ref{lem:Rtdisjoint} and the equivalence supplied by
    Lemma~\ref{lem:stage},
    \begin{align*}
    R_t(\E_{n+1}\sqcup\M_{n+1},\ \dim\sqcup f_{n+1})
    &=R_t(\E_n,\dim)+R_t(\E\sqcup\M',\ \dim\sqcup g')\\
    &=R_t(\E_n\sqcup\M_n,\ \dim\sqcup f_n).
    \end{align*}
    Together with the inductive hypothesis and the transitivity of $\sim$, this gives
    \[
    (\K,f)\sim
    (\E_{n+1}\sqcup\M_{n+1},\ \dim\sqcup f_{n+1}).
    \]

    By Lemma~\ref{lem:stage}, $\E$ is a disjoint union of finitely many
    elementary blocks, each consisting of one $n$-cell and one $(n+1)$-cell.
    Hence $\E_{n+1}\setminus\E_n=\E$ has finitely many cells, all of
    dimension $n$ or $n+1$.

    The equality of Poincar\'{e} series follows by induction on $n$.
    For $n=0$, it holds because $\M_0=\K$ and the sum is empty.
    Assume that it holds for $n$. By Lemma~\ref{lem:stage},
    \[
    P_t(\M_n)=P_t(\M_{n+1})+b_n^{\M_n}\,t^n.
    \]
    Therefore,
    \begin{align*}
    P_t(\K)
    &=P_t(\M_n)+\sum_{m=0}^{n-1}b_m^{\M_m}\,t^m\\
    &=P_t(\M_{n+1})+\sum_{m=0}^{n}b_m^{\M_m}\,t^m,
    \end{align*}
    which proves the equality for $n+1$.
\end{proof}

We are now ready to prove the elementary block reduction theorem stated in the
Introduction.

\begin{proof}[Proof of Theorem~\ref{thm:elementary}]
    Let $\E_n$, $\M_n$, and $f_n$ be as in Proposition~\ref{prp:stage}, and
    put $\E:=\bigcup_{n\ge0}\E_n$. The cells of $\E_{n+1}\setminus\E_n$ have dimension $n$ or
    $n+1$, so $\dim^{-1}(k)$ is finite for every $k$ and $\E$ is a Lefschetz complex, which is
    a disjoint union of elementary blocks.

    Fix $k\in\Z_{\ge0}$ and put $n:=k+2$. All cells of $\M_n$ have dimension at least $n$, so
    $C_k(\M_n)=C_{k+1}(\M_n)=\{0\}$, and therefore $\rank B_k(\M_n)=0$,
    $d^\snc_{k+1}(\M_n,f_n)=0$ and $\rank B_k(\W(\M_n,f_n))=0$, that is, the coefficient
    $R_t(\M_n,f_n)_k$ of $t^k$ vanishes.
    By Proposition~\ref{prp:stage}, Lemma~\ref{lem:Rtdisjoint} and Lemma~\ref{lem:dimMB},
    \[
    R_t(\K,f)_k=R_t(\E_n,\dim)_k+R_t(\M_n,f_n)_k=\rank B_k(\E_n).
    \]
    Finally, the cells of $\E\setminus\E_n$ have dimension at least $n$, so
    \[
    C_k(\E)=C_k(\E_n)
    \qquad\text{and}\qquad
    C_{k+1}(\E)=C_{k+1}(\E_n).
    \]
    Therefore,
    \[
    \rank B_k(\E)=\rank B_k(\E_n).
    \]
    It follows that
    $R_t(\K,f)_k=\rank B_k(\E)=R_t(\E,\dim)_k$ for every $k$.
\end{proof}

\begin{corollary}\label{cor:finitemoves}
    If $\K$ has finitely many cells, then $\E$ is obtained from $\K$ by finitely many Moves.
\end{corollary}

\begin{proof}
    Consider the complex $\E$ constructed in the proof of
    Theorem~\ref{thm:elementary}. The assertion is immediate if $\K=\emptyset$.
    Otherwise, let $N$ be the largest
    dimension of a cell of $\K$. Since every cell of $\M_{N+1}$ has
    dimension at least $N+1$, we have $\M_{N+1}=\emptyset$ and hence $\E=\E_{N+1}$. Each
    application of Lemma~\ref{lem:stage} in Proposition~\ref{prp:stage} consists of finitely
    many Moves, and only $N+1$ such applications are needed.
\end{proof}

\begin{corollary}\label{cor:count}
    In the construction of Proposition~\ref{prp:stage}, the following statements hold for
    every $k\in\Z_{\ge0}$:
    \begin{enumerate}[(1)]
        \item $R_t(\K,f)_k$ equals the number of elementary blocks $E(\mu,\kappa)$ with
              $\dim\mu=k+1$ produced by Move~(IV),
        \item $b_k^\K$ equals the number of $k$-cells removed by Move~(II).
    \end{enumerate}
    In particular, the coefficients of $R_t(\K,f)$ are nonnegative integers, and both
    counts are independent of the choices made in the construction.
\end{corollary}

\begin{proof}
    (1) By Theorem~\ref{thm:elementary}, Lemma~\ref{lem:Rtdisjoint} and
    Lemma~\ref{lem:Rtblock},
    \[
    R_t(\K,f)
    =R_t(\E,\dim)
    =\sum_i R_t(E(\mu_i,\kappa_i),\dim)
    =\sum_i t^{\dim\mu_i-1},
    \]
    where $\E$ is the disjoint union of the elementary blocks $E(\mu_i,\kappa_i)$ produced
    by Move~(IV). Therefore, $R_t(\K,f)_k$ is the number of these blocks with
    $\dim\mu_i=k+1$.

    (2) Fix $k$. By Proposition~\ref{prp:stage} with $n=k+1$,
    \[
    P_t(\K)
    =P_t(\M_{k+1})+\sum_{m=0}^{k}b_m^{\M_m}t^m.
    \]
    Since every cell of $\M_{k+1}$ has dimension at least $k+1$, we have
    $b_k^{\M_{k+1}}=0$. Comparing the coefficients of $t^k$ gives
    \[
    b_k^\K=b_k^{\M_k}.
    \]
    By Lemma~\ref{lem:stage}, this is the number of $k$-cells removed by Move~(II) at
    stage $k$.
\end{proof}

We are now ready to prove the Morse-Bott inequality stated in the Introduction.

\begin{proof}[Proof of Theorem~\ref{thm:MB}]
    The identity follows from Lemma~\ref{lem:MBpoly}, and the assertion about the
    coefficients follows from Corollary~\ref{cor:count}.
\end{proof}

\begin{corollary}[Morse-Bott inequality for CW complexes]\label{cor:MBCW}
    Let $M$ be a finite CW complex, and let
    $\K_M$ be its associated Lefschetz complex as in Example~\ref{ex:CWasLef}.
    If $f$ is a discrete Morse-Bott function on $M$ in the sense
    of~\cite{nishikawa2025DMBT}, then $f$ is also a discrete Morse-Bott function on
    $\K_M$, and
    \[
        R_t(M,f)=R_t(\K_M,f).
    \]
    In particular, $R_t(M,f)$ has nonnegative integer coefficients.
\end{corollary}

\begin{proof}
    The CW complex $M$ and its associated Lefschetz complex $\K_M$ have the same cells,
    dimension function, and incidence numbers. Hence the sets
    $U^{\snc}(\sigma)$ and $D^{\snc}(\sigma)$ computed in $M$ agree with those computed
    in $\K_M$ for every cell $\sigma$. Thus, $f$ is a discrete Morse-Bott function on
    $\K_M$, and the weakly critical sets determined by $f$ are the same in both
    settings. Moreover, the cellular chain complex of $M$ is precisely the chain
    complex associated with $\K_M$. Consequently, the Betti numbers of $M$ and of
    each of its weakly critical sets agree with the corresponding Betti numbers
    computed as Lefschetz complexes. It follows from the definitions of the two
    remainder terms that
    \[
        R_t(M,f)=R_t(\K_M,f).
    \]
    By Theorem~\ref{thm:MB}, $R_t(\K_M,f)$ has nonnegative integer coefficients,
    and hence so does $R_t(M,f)$.
\end{proof}

\begin{corollary}[Morse inequality for Poincar\'{e} series]\label{cor:M}
    Let $f$ be a discrete Morse function on $\K$. Then the formal power series $r_t(\K,f)$
    has nonnegative integer coefficients and
    \[
    \sum_{k=0}^\infty m_k t^k=P_t(\K)+(1+t)r_t(\K,f).
    \]
\end{corollary}

\begin{proof}
    The identity follows from Lemma~\ref{lem:poly}. By Remark~\ref{rem:MtoMB},
    we have $r_t(\K,f)=R_t(\K,f)$, so the assertion about the coefficients
    follows from Theorem~\ref{thm:MB}.
\end{proof}

\section*{Acknowledgments}
The author is deeply grateful to Professor Tomoo Yokoyama for his valuable suggestions and helpful advice.

\section*{Declaration on the use of generative AI}
During the preparation of this manuscript, the author used Claude (Anthropic) and
ChatGPT (OpenAI) to assist with English-language drafting and editing, manuscript
organization, and bibliographic searches. The author reviewed and verified all
AI-assisted material, independently checked the cited sources, and takes full
responsibility for the content of the manuscript.

\bibliographystyle{abbrv}
\bibliography{DMBFonLC}

\end{document}